\documentclass[11pt]{article}
\usepackage{amsmath,amsthm,amsfonts,amsthm,amssymb, mathtools,longtable,multirow,subfig}
\usepackage{mathrsfs,array,graphicx,color,latexsym,tikz,calc,pifont,natbib}
\usetikzlibrary{shadows,patterns,arrows,decorations.pathreplacing}
\usetikzlibrary{calc,arrows.meta,arrows,decorations.markings,chains,positioning}

\setcitestyle{numbers,square}
\DeclareMathOperator{\dv}{div}
\DeclareMathOperator{\diag}{diag}
\DeclareMathOperator{\curl}{curl}
\DeclareMathOperator{\grad}{grad}
\DeclareMathOperator{\Sp}{Sp}
\DeclareMathOperator{\Line}{Line}
\DeclareMathOperator{\sgn}{sgn}

\theoremstyle{plain}
\newtheorem{thm}{\bf Theorem}[subsection]
\newtheorem{cor}[thm]{\bf Corollary}

\theoremstyle{definition}
\newtheorem{re}[thm]{\bf Remark}
\newtheorem{lem}[thm]{\bf Lemma}
\newtheorem{prop}[thm]{\bf Proposition}
\newtheorem{problem}[thm]{\bf Problem}
\newtheorem{examplex}[thm]{\bf Example}

\newcommand{\field}[1]{\mathbb{#1}}
\newcommand{\Z}{\field{Z}}

\renewenvironment{abstract}{\par\noindent\textbf{\abstractname.}\ \ignorespaces}{\par\medskip}
\providecommand{\keywords}[1]{\small \textbf{\textit{Keywords.}} #1}
\providecommand{\ams}[1]{\small \textbf{\textit{AMS subject classifications.}} #1}

\title{\bf \Large Helmholzian Spectra of Graphs:  Novel Properties\vspace{-0.5em}}

\author{
{\small   Lu Lu$^a$, \ \ Yongtang Shi$^b$, \ \ Zoran Stani\'c$^{c}$, \ \ Jianfeng Wang$^{d,}$\footnote{Corresponding author.\newline{Email addresses:} lulugdmath@163.com (L. Lu), shi@nankai.edu.cn (Y.T. Shi), zstanic@matf.bg.ac.rs (Z. Stani\' c), jfwang@aliyun.com (J.F. Wang), wangy@ahu.edu.cn (Y. Wang).}\;, \ \ Yi Wang$^e$ }\\[2mm]
\footnotesize $^a$School of Mathematics and Statistics, HNP-LAMA, Central South University,
Changsha 410083, China\\
\footnotesize $^b$Center for Combinatorics and LPMC, Nankai University, Tianjin 300071, China\\
\footnotesize $^c$Faculty of Mathematics, University of Belgrade, Studentski trg 16, Belgrade, Serbia\\
\footnotesize $^d$School of Mathematics and Statistics, Shandong University of Technology, Zibo 255049, China\\
\footnotesize $^e$School of Mathematics, Anhui University, Hefei 230039, China\\}

\date{}

\begin{document}
\maketitle

\begin{abstract}
Let $\grad$, $\curl$, and $\dv$ be the graph-theoretic analogues of the gradient, curl, and divergence operators from multivariate calculus. The graph Laplacian $-\dv \grad$ gives rise to the celebrated Laplacian matrix, while the matrix representation of the graph Helmholtzian $\grad \grad^* + \curl^* \curl$ is called the Helmholtzian matrix. In this paper, we present a new graph-theoretic proof that the Helmholtzian matrix indeed represents the graph Helmholtzian. We then investigate the spectral properties of this matrix. Our main results are as follows: (i) a classification of graphs having exactly two distinct Helmholtzian eigenvalues; (ii) the nullity of the Helmholtzian matrix; and (iii) a combinatorial interpretation of the coefficients of the Helmholtzian polynomial. Furthermore, we determine the Helmholtzian spectrum for certain graph products and characterize Helmholtzian integral graphs, as well as derive bounds for the smallest Helmholtzian eigenvalue. Meanwhile, we pose some open problems for future research.\\[2mm]
\end{abstract}

{\noindent}\keywords Helmholtzian matrix; Graph spectra; Signed graph; Hodge Laplacians; Small world network\\

{\noindent}\ams 05C50, 39A12, 05C22, 05C20, 05C82.

%\tableofcontents

\section{Introduction}

In the majority of the paper the considered graphs are undirected and simple. For such a graph $G = (V(G),E(G))$,  two vertices $u$ and $v$ are {\it adjacent} if the pair $u, v$ features as an edge (that is, $uv\in E(G)$), and these vertices are recognized as {\it neighbours}. The {\it degree} of a vertex $v$, denoted by $d_G(v)$, is the number of its neighbours; in this and similar notations, we suppress the subscript $G$ whenever it is clear from the context. Let $M(G)$ be a prescribed matrix associated with $G$. The {\it $M$-eigenvalues} of $G$ are the eigenvalues of $M$, and the {\it $M$-spectrum} of $G$ is the multiset of its $M$-eigenvalues. At this moment, we restrict ourselves to the adjacency matrix and the Laplacian matrix. For definitions of other graph matrices that  sporadically appear in the forthcoming text, we refer the reader to \cite{cve-row-sim-book,cve-doob-sachs-book,ifge}. The \textit{adjacency matrix}, denoted by $A(G) = (a_{ij})$, has rows and columns indexed by the vertices, along with $a_{ij} =1$ if $v_iv_j \in E(G)$ and $a_{ij} =0$ otherwise. The \textit{Laplacian matrix} is defined as $L(G) = D(G)-A(G)$ where $D(G) = \diag(d(v_1),d(v_2),\ldots,d(v_n))$; the diagonal matrix of vertex degrees.

As is well-known, the spectral graph theory plays an important  role in many areas including computer science, chemistry, physics, biology, electrical engineering, along with many branches of mathematics \cite{cve-row-sim-book,cve-doob-sachs-book,cve-sim,dam-koo,ran-nov-pla}. It is also recognized in contemporary disciplines including neural networks~\cite{wu-IEEE},  complex networks \cite{mie}, social media~\cite{Zas}, and others. For examples, the original algorithm `PageRank' of Google's search engine was designed on the basis of spectra of the so-called Google matrices that are actually matrices associated with directed graphs. Recently, the Sensitivity Conjecture  in theoretical computer science was completely resolved in terms of spectra of the adjacency matrix of the so-called signed graphs~\cite{huang}. The problem on equiangular lines in prescribed Euclidean spaces was settled with a key strategy based on the spectrum of the adjacency matrix of associated graphs having a sublinear  multiplicity of the second largest eigenvalue~\cite{jiang-zhao}. To solve Chowla conjecture in the Number theory, Helfgott and Radziwi{\l}{\l} \cite{hel-rad} have modified specified graphs constructed in \cite{mrtao,tao} and link them to graphs with a high expanding property relative to the adjacency matrix. In short, different graph spectra play different roles when studying different problems. For background on Hodge theory and networks, see Section 4.

In this paper, we consider a certain generalization of the Laplacian matrix.  For our goals, we firstly introduce three operators on functions on a graph $G$.  A {\it $k$-clique} of $G$ is the vertex subset
$$
\mathcal{K}_k(G)=\{v_{i_1},v_{i_2},\ldots,v_{i_k}\} \quad \mbox{such that} \quad v_{i_p}v_{i_q} \in E(G),\quad \mbox{for}\quad 1 \leq p < q  \leq k.
$$
For convince, set $\mathcal{K}_0= V(G) \coloneqq V$, $\mathcal{K}_1 = E(G) \coloneqq E$ and $\mathcal{K}_2 = T(G) \coloneqq T$, the set of 3-cliques or triangles. A real function $f$ on $V$ is $f: V \rightarrow \mathbb{R}$. An alternation function $X: V\times V \rightarrow \mathbb{R}$ is a real function on $E$ such that
$$
X(i,j)=
\begin{cases}
-X(j,i), & \mbox{if $\{v_i,v_j\}\in E$}\\
0, & \mbox{otherwise}.
\end{cases}
$$
An alternating function $\Phi: V \times V \times V \rightarrow \mathbb{R}$ is a real function on $T$ such that
$$
\Phi(i,j,k)=\Phi(j,k,i)=\Phi(k,i,j)=
\begin{cases}
-\Phi(j,i,k)=-\Phi(i,k,j)=-\Phi(k,j,i), & \mbox{if $\{v_i,v_j,v_k\}\in T$}\\
0, & \mbox{otherwise}.
\end{cases}
$$
Equipping  the inner product on functions $f$, $X$ and $\Phi$:
\[\langle f,g\rangle_V=\sum_{i\in V}f(i)g(i),\;\langle X,Y\rangle_E=\sum_{i\le j}X(i,j)Y(i,j) \;\; {\mbox{and}} \;\;
\langle \Phi,\Psi\rangle_T=\sum_{i<j<k}\Phi(i,j,k)\Psi(i,j,k),\]
we obtain the corresponding Hilbert spaces $L^2(V)$, $L^2_{\wedge}(E)$ and
$L^2_{\wedge}(T)$, respectively. Then  the {\it gradient} operator  $\grad\colon L^2_{\wedge}(V)\longrightarrow
L^2_{\wedge}(E)$ is defined as
$$(\grad~f)(i,j)=
\begin{cases}
f(j)-f(i) & \mbox{if $\{v_i,v_j\}\in E$},\\
 0, & \mbox{otherwise}.
\end{cases}
$$
The {\it curl} operator $\curl\colon L^2_{\wedge}(E)\longrightarrow L^2_{\wedge}(T)$ is defined as
$$
(\curl~X)(i,j,k)=
\begin{cases}
X(i,j)+X(j,k)+X(k,i) & \mbox{if $\{v_i,v_j,v_k\}\in T$},\\
0, & \mbox{otherwise}.
\end{cases}
$$
The {\it divergence} operators $\dv\colon L^2_{\wedge}(E) \rightarrow  L^2(V)$ is defined as $$(\dv~X)(i) =\sum^n_{j=1}\frac{w_{ij}}{w_i}X(i,j)$$ for all $i \in  V$. Note that the above three operators are respectively the graph-theoretic analogues of $\grad,\curl$ and $\dv$ in multivariate calculus; while the functions $f$, $X$ and $\Phi$ are called 0-, 1-, 2-cochains in topological parlance, which are discrete analogues of differential forms on manifolds \cite{warner}.

As known, the graph Laplacian $\Delta_0 =  -\dv\,\grad$ is a graph-theoretic analogue of Laplace operator (or scalar Laplacian) operating a scalar function, which gives us the celebrated Laplacian matrix $L(G)$ (see \cite[eg.]{Hodge-lim}). Emphatically, our concern is the graph Helmholtzian defined as
\begin{equation*}\label{hodge-H1}
\mathcal{H}\coloneqq\Delta_1= \grad\; \grad^* + \curl^*\; \curl.
\end{equation*}
which is a graph-theoretic analogue of Helmholtz operator (or vector Laplacian) operating  a vector field. Later, the matrix representation $\mathcal{H}(G)$ of graph Helmholtzian was expressed as the following sum of two matrices
\begin{equation}\label{H=B+C}
\mathcal{H}(G)={\mathcal B}(G){\mathcal B}(G)^{\intercal} + {\mathcal C}(G)^{\intercal}{\mathcal C}(G),
\end{equation}
where ${\mathcal B}$ and ${\mathcal C}$ are defined in \cite[Section 5.3]{Hodge-lim} (or see Section \ref{H-matrix1}). Recently, the combinatorial expression of this matrix $\mathcal{H}(G)$, termed the {\it Helmholtzian matrix}, was derived from an algebraic topology perspective. Unlike the usual vertex-indexed graph matrices, this is an edge-indexed matrix. On this basis, the fundamental properties of Helmholtzian matrix were studied \cite{li-lu-wang}.

This paper is organized as follows. From a graph-theoretic perspective, we present a new proof in Section~\ref{H-matrix1} that the Helmholtzian matrix represents the graph Helmholtzian. In Section 3, we explore the spectral characteristics of Helmholtzian matrix through the six subsections, the novel results of which are summarized as follows:
\begin{itemize}
  \item[(i)]
A classification of graphs with exactly two distinct Helmholtzian eigenvalues;
  \item[(ii)]
The nullity of Helmholtzian matrix;
  \item[(iii)]
A combinatorial interpretation of coefficients of the Helmholtzian polynomial.
\end{itemize}
In addition, we establish the Helmholtzian spectrum of  prescribed graph products, some integral graphs (i.e., graphs whose Helmholtzian spectrum consists entirely of integers) and bounds for the least Helmholtzian eigenvalue. Finally, some problems for further study are posed in the related sections. Finally, more motivations and potential applications are introduced in Section 4.

\section{A Graph-theoretic Proof for Helmholtzian Matrix}\label{H-matrix1}

In order to introduce the  Helmholtzian matrix of a graph, we first need some necessary terminology and notation. For any orientation on the edges and triangles of an ordinary graph $G$, the {\it tail} and the {\it head} of an oriented edge $e$ are designated by $e^-$ and $e^+$, respectively. We set $\mathcal{V}(e)=\{e^-,e^+\}$. If there is an oriented edge from $u$ to $v$, then we write
$u\rightarrow v$. If  $u$ and $v$ are adjacent, then we write $u\sim v$, and $u
\nsim v$ otherwise.  Therefore, $u\sim v$
implies either $u\rightarrow v$ or $v\rightarrow u$. If a vertex $v$ satisfies $v\in
\mathcal{V}(e)$, then we write $v \in e$. Moreover, we set $v\rightarrow e$ if $v=e^-$, and
$e\rightarrow v$ if $v=e^+$. For two edges $e_1$ and $e_2$, we write $e_1\sim e_2$ if
$\mathcal{V}(e_1)\cap
\mathcal{V}(e_2)\ne\emptyset$. Also, we write
$e_1\rightarrow e_2$ if $e_1^+=e_2^-$, $e_1\overset{+}{\sim}e_2$ if $e_1^+=e_2^+$, and
$e_1\overset{-}{\sim}e_2$   if $e_1^-=e_2^-$. Further, $e_1\leftrightarrow e_2$ means that
either $e_1\rightarrow e_2$ or $e_2 \rightarrow e_1$,   $e_1\overset{\pm}{\sim}e_2$ means that
either $e_1\overset{+}{\sim}e_2$ or $e_1\overset{-}{\sim}e_2$, and  $e_1\vartriangle e_2$ means that $e_1$
and $e_2$ are in the same triangle. For an edge $e$ and a triangle
$\vartriangle$, we write $e\in\vartriangle$ if $e$ is an edge of $\vartriangle$. Moreover, if
the orientation of $e$ agrees with that of $\vartriangle$, then we write
$e\in\vartriangle^+$, and $e\in\vartriangle^-$ otherwise. To make the symbols more clear, we
collect them in Table~\ref{tab-1}; they will be frequently used in the entire paper.

\begin{table}[h]
		\centering
		\begin{minipage}{0.48\textwidth}
			\centering
			\begin{tabular}{|c|c|c|c|}
				\hline
				\multicolumn{3}{|c|}{Symbol} & Diagram \\
				\hline
				\multirow{2}*{$u\sim v$} & \multicolumn{2}{c|}{$u\rightarrow v$} &
				\includegraphics[scale=0.6]{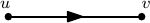}  \\[1.6mm]
				\cline{2-4}
				~ & \multicolumn{2}{c|}{$v\rightarrow u$} & \includegraphics[scale=0.6]{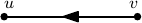}  \\[1.6mm]
				\hline
				\multirow{2}*{$u\in e$} & \multicolumn{2}{c|}{$u\rightarrow e$~($u=e^-$)} &
				\includegraphics[scale=0.6]{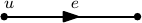}  \\[1.6mm]
				\cline{2-4}
				~ & \multicolumn{2}{c|}{$e\rightarrow u$~($u=e^+$)} & \includegraphics[scale=0.6]{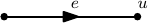}  \\[1.6mm]
				\hline
				\multirow{4}*{$e_1\sim e_2$} & \multirow{2}*{$e_1\overset{\pm}{\sim}e_2$} &
				$e_1\overset{+}{\sim}e_2$ & \includegraphics[scale=0.6]{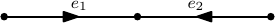}  \\[1.6mm]
				\cline{3-4}
				~ & ~ & $e_1\overset{-}{\sim}e_2$ & \includegraphics[scale=0.6]{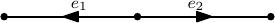} \\[1.6mm]
				\cline{2-4}
				~& \multirow{2}*{$e_1\leftrightarrow e_2$} & $e_1\rightarrow e_2$ &
				\includegraphics[scale=0.6]{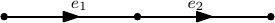}\\[1.6mm]
				\cline{3-4}
				~ & ~ & $e_2\rightarrow e_1$ &  \includegraphics[scale=0.6]{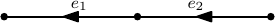}\\[1.6mm]
				\hline
			\end{tabular}
		\end{minipage}
		\hfill
		\begin{minipage}{0.48\textwidth}
			\centering
			\begin{tabular}{|c|c|c|c|}
				\hline
				\multicolumn{3}{|c|}{Symbol} & Diagram \\
				\hline
				\multicolumn{3}{|c|}{$e_1\vartriangle e_2$} &
				\raisebox{-.5\height}{\includegraphics[scale=0.65]{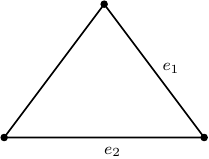}}\\[1.6mm]
				\hline
				\multirow{2}*{$e\in\vartriangle$} & \multicolumn{2}{c|}{$e\in\vartriangle^+$} &
				\raisebox{-.5\height}{\includegraphics[scale=0.75]{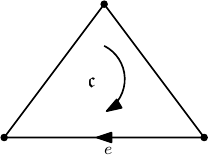}}\\[1.6mm]
				\cline{2-4}
				~ & \multicolumn{2}{c|}{$e\in\vartriangle^-$} &
				\raisebox{-.5\height}{\includegraphics[scale=0.75]{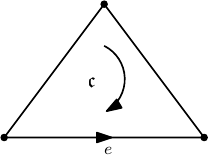}} \\[1.6mm]
				\hline
			\end{tabular}
		\end{minipage}
\caption{Relations between vertices, edges and triangles.}\label{tab-1}
	\end{table}

For each vertex $u \in V(G)$, the {\it in-neighbourhood}, the {\it out-neighbourhood} and
the {\it neighbourhood}  are  defined as
$$N_G^+(u)=\{v\in V\mid v\rightarrow u\},\, N_G^-(u)=\{v\in V\mid u\rightarrow v\}~
\text{and}~  N_G(u)=N_G^+(u)\cup N_G^-(u),$$ respectively. The cardinalities $d_G^+(u)=|N_G^+(u)|$,
$d_G^-(u)=|N_G^-(u)|$ and $d_G(u)=|N_G(u)|$ are called the {\it in-degree}, the {\it
out-degree} and the {\it degree} of $u$. For an edge $e \in E(G)$, the {\it triangle degree of $e$}, denoted by $\triangle_G(e),$ is the number of triangles containing $e$, that is,
$$\triangle_G(e)=|\{\vartriangle\in T(G)\mid e\in \vartriangle\}|,$$ which is also denoted by $\triangle(e)$ when $G$ is unambiguous. As usual, we
write $O_{x\times y}$, $J_{x\times y}$, $I_{x}$, $\mathbf{0}_x$ $\mathbf{1}_x$ and  for the all-zero
matrix, the all-one matrix, the identity matrix, the all-0 column vector and the all-1 column vector of the corresponding size, respectively.
The subscript may be omitted. The {\it edge-vertex incidence matrix}
$\mathcal{B}=\mathcal{B}(G)=(b_{ev})_{m\times n}$ and the {\it triangle-edge incidence matrix}
$\mathcal{C}=\mathcal{C}(G)=(c_{\vartriangle e})_{t\times m}$ are defined as
\begin{equation*}\label{B-C}
b_{ev}=
\begin{cases}
-1,& v\rightarrow e\\
1,&e\rightarrow v\\
0,&\text{otherwise}
\end{cases} \;\;\text{and}\;\; c_{\vartriangle e}=
\begin{cases}-1,&
e\in\vartriangle^-\\
1,&e\in\vartriangle^+\\
0,&\text{otherwise}.
\end{cases}
\end{equation*}

\begin{lem}\label{lem-1}
The operator $\grad\grad^*$ is represented by the matrix ${\mathcal B}{\mathcal B}^{\intercal}$, and the operator $\curl^*\curl$ is represented by the matrix ${\mathcal C}^{\intercal}{\mathcal C}$.
\end{lem}

\begin{proof}
Assign $E(G)$ and $T(G)$ arbitrary orientations. For any $v\in V(G)$, let $\delta_v\in
L^2_{\wedge}(V)$ be the function, such that $\delta_v(x)=1$ if $x=v$ and $0$ otherwise. For any edge
$e\in E(G)$, let $\delta_{e}\in L^2_{\wedge}(E)$ denote the function such that
$$\delta_e(i,j)=-\delta_e(j,i)=1$$ if $e=\{i,j\}$ and $0$ otherwise. For any
$\vartriangle\in T(G)$, set $\delta_{\vartriangle}\in L^2_{\wedge}(T)$ to be the function such that
$$\delta_{\vartriangle}(i,j,k)=\delta_{\vartriangle}(j,k,i)=\delta_{\vartriangle}(k,i,j)
=-\delta_{\vartriangle}(j,i,k)=-\delta_{\vartriangle}(i,k,j)=-\delta_{\vartriangle}(k,j,i)=1$$
if $\vartriangle= \{i,j,k\}$ and $0$ otherwise. Clearly, $\{\delta_v\mid v\in V(G)\}$,
$\{\delta_e\mid e\in E(G)\}$ and $\{\delta_{\vartriangle}\mid\vartriangle\in T(G)\}$ are
orthonormal basis of $L^2_{\wedge}(V)$, $L^2_{\wedge}(E)$ and $L^2_{\wedge}(T)$, respectively. Assume that ${\rm grad}$
gives the matrix $\mathcal{X}$ and ${\rm curl}$ gives the matrix ${\mathcal{Y}}$ under these basis. It suffices to
show that $\mathcal{X}=\mathcal{B}$ and $\mathcal{Y}={\mathcal C}$, since $f^{\ast}$ gives the matrix $F^*$ if $f$ gives the matrix~$F$.

Note that the $(e,v)$-entry of $\mathcal{X}$ is $$\mathcal{X}_{e,v}=\langle
\operatorname{grad}\delta_v,\delta_e\rangle_E=\sum_{i\le
j}(\operatorname{grad}\delta_v)(i,j)\delta_e(i,j)=\delta_v(e^+)-\delta_v(e^-),$$ which yields that
$\mathcal{X}_{e,v}=1$ if $v=e^+$, $-1$ if $v=e^-$ and $0$ otherwise, and therefore $\mathcal{X}=\mathcal{B}$.

Similarly, the $(\vartriangle,e)$-entry of $\mathcal{Y}$ is $$\mathcal{Y}_{\vartriangle,e}=\langle
{\rm curl}~\delta_e,\delta_{\vartriangle}\rangle_T=\sum_{x<y<z}({\rm curl}~\delta_e)(x,y,z)
\delta_{\vartriangle}(x,y,z) =\delta_e(i,j)+\delta_e(j,k)+\delta_e(k,i),$$ where
$\vartriangle=(i,j,k)$. Hence, if $e\in \vartriangle^+$ then $e\in\{(i,j),(j,k),(k,i)\}$,
and thus $\mathcal{Y}_{\vartriangle,e}=1$; if $e\in\vartriangle^-$ then $e\in\{(j,i),(k,j),(i,k)\}$, and
thus $\mathcal{Y}_{\vartriangle,e}=-1$; if $e\not\in\vartriangle$ then $\mathcal{Y}_{\vartriangle,e}=0$. Thereby,
$\mathcal{Y}=\mathcal{C}$.
\end{proof}

Note, Li and two authors of this paper \cite{li-lu-wang} obtained the following matrix representation of graph Helmholtzian by the {\it combinatorial Laplacian matrix} defined on simplicial complex, whose properties are studied from the perspective of algebraic topology by Goldberg \cite{god-roy}. For the self-containedness of the paper, we  provide a graph-theoretic proof.

\begin{thm}\label{H-def}
Let $G$ be a graph with orientations on its edge set $E$ and triangle set $T$. Then the Helmholtzian matrix $\mathcal{H}(G)=(h_{ee'})$ of $G$ is defined as follows:
$$
h_{ee'}=
\begin{cases}
\triangle(e)+2, & \mbox{if $e'=e$}\\
-1, & \mbox{if $e\leftrightarrow e'$ and $e\not\vartriangle e'$}\\
1, & \mbox{if $e'\overset{\pm}{\sim} e$ and $e\not\vartriangle e'$} \\
0, & \mbox{otherwise}.
\end{cases}
$$
\end{thm}

\begin{proof}
By employing Lemma \ref{lem-1} and \eqref{H=B+C}, a direct computation shows that the $(e,e')$-entry of $\mathcal{H}(G)$ is given as
\[\begin{array}{lll}
\mathcal{H}_{e,e'}&=&(\mathcal{B}\mathcal{B}^{\intercal})_{e,e'}+(\mathcal{C}^{\intercal}\mathcal{C})_{e,e'}\\[2mm]
&=&\sum_{v\in V(G)}\mathcal{B}_{e,v}\mathcal{B}^{\intercal}_{v,e'}+\sum_{\vartriangle\in T(G)}(\mathcal{C}^{\intercal})_{e,\vartriangle}\mathcal{C}_{\vartriangle,e'}\\[2mm]
&=&\sum_{v\in V(G)}\mathcal{B}_{e,v}\mathcal{B}_{e',v}+\sum_{\vartriangle\in T(G)}\mathcal{C}_{\vartriangle,e}\mathcal{C}_{\vartriangle,e'}\\[2mm]
&=&\sum_{v\in \mathcal{V}(e)\cap \mathcal{V}(e')}\mathcal{B}_{e,v}\mathcal{B}_{e',v}+\sum_{e,e' \in \vartriangle}\mathcal{C}_{\vartriangle,e}\mathcal{C}_{\vartriangle,e'}.
\end{array}\]
If $e=e'$, then $$\sum_{v\in \mathcal{V}(e)\cap
\mathcal{V}(e')}\mathcal{B}_{e,v}\mathcal{B}_{e',v}=\sum_{v\in \mathcal{V}(e)}\mathcal{B}_{e,v}^2=2 \;\; \mbox{and}\;\sum_{e,e' \in \vartriangle
}\mathcal{C}_{\vartriangle,e}\mathcal{C}_{\vartriangle,e'}=\sum_{e \in \vartriangle}\mathcal{C}_{\vartriangle,e}^2=\triangle(e),$$
which results in $\mathcal{H}_{e,e'}=\triangle(e)+2$. If $e\not\sim e'$, then
$$\sum_{v\in \mathcal{V}(e)\cap \mathcal{V}(e')}\mathcal{B}_{e,v}\mathcal{B}_{e',v}=0 \;\; \mbox{and}\;\sum_{e,e' \in \vartriangle}\mathcal{C}_{\vartriangle,e}\mathcal{C}_{\vartriangle,e'}=0,$$
which shows $\mathcal{H}_{e,e'}=0$. It remains to consider the case $e\sim e'$. If $e\vartriangle e'$
and $e\leftrightarrow e'$, say $e,e'\in\vartriangle_0$ and $e\rightarrow e'$, then
$$\sum_{v\in \mathcal{V}(e)\cap \mathcal{V}(e')}\mathcal{B}_{e,v}\mathcal{B}_{e',v}=\mathcal{B}_{e,e^+}\mathcal{B}_{e',{e'}^-}=-1\;\;
\mbox{and}\;\sum_{e,e' \in \vartriangle
}\mathcal{C}_{\vartriangle,e}\mathcal{C}_{\vartriangle,e'}=\mathcal{C}_{\vartriangle_0,e}
\mathcal{C}_{\vartriangle_0,e'}=1$$
 indicating  $\mathcal{H}_{e,e'}=0$. If $e \vartriangle e'$ and $e\overset{\pm}{\sim} e'$, say
$e,e'\in\vartriangle_1$ and $e\overset{+}{\sim}e'$, then
$$\sum_{v\in \mathcal{V}(e)\cap \mathcal{V}(e')}\mathcal{B}_{e,v}\mathcal{B}_{e',v}=\mathcal{B}_{e,e^+}\mathcal{B}_{e',{e'}^+}=1\;\;
\mbox{and}\sum_{e,e' \in \vartriangle
}\mathcal{C}_{\vartriangle,e}\mathcal{C}_{\vartriangle,e'}=\mathcal{C}_{\vartriangle_1,e}
\mathcal{C}_{\vartriangle_1,e'}=-1,$$ which yields
$\mathcal{H}_{e,e'}=0$. If $e\not\vartriangle e'$ and $e \leftrightarrow e'$, say $e\rightarrow e'$,
then $$\sum_{v\in \mathcal{V}(e)\cap \mathcal{V}(e')}\mathcal{B}_{e,v}\mathcal{B}_{e',v}=\mathcal{B}_{e,e^+}\mathcal{B}_{e',{e'}^-}=-1\;\; \mbox{and}\sum_{e,e' \in \vartriangle}\mathcal{C}_{\vartriangle,e}\mathcal{C}_{\vartriangle,e'}=0,$$ which arrives at
$\mathcal{H}_{e,e'}=-1$. If $e\not\vartriangle e'$ and $e\overset{\pm}{\sim} e'$, say
$e\overset{+}{\sim}e'$, then
$$\sum_{v\in \mathcal{V}(e)\cap \mathcal{V}(e')}\mathcal{B}_{e,v}\mathcal{B}_{e',v}=\mathcal{B}_{e,e^+}\mathcal{B}_{e',{e'}^+}=1\;\; \mbox{and}\;
\sum_{e,e' \in \vartriangle}\mathcal{C}_{\vartriangle,e}\mathcal{C}_{\vartriangle,e'}=0$$ implying
$\mathcal{H}_{e,e'}=1$.

The proof is completed.
\end{proof}

As we said, the matrix $\mathcal{H}(G)$ is called the  Helmholtzian matrix of  $G$. In contrast to striking ways from all other known graph matrices, this matrix is indexed by graph edges. By a convention, an edgeless graph is represented by an empty Helmholtzian (i.e., the $0\times 0$ matrix). Throughout the entire paper, unless told otherwise, we assume that a graph under consideration has at least one edge!

At present, some basic properties of Helmholtzian matrix have been obtained in \cite{li-lu-wang}, including the irreducibility and positive semi-definiteness of this matrix, eigenvalue interlacing  under edge additions and the relationships with the spectrum of the adjacency matrix and the Laplacian matrix of a related oriented graph, weighted graph or signed graph. Particularly, although the matrix $\mathcal{H}(G)$ leans on the orientation of edges, the eigenvalues of Helmholtzian matrix do not depend on the choice of orientation
\cite[Theorem 2.1]{lu-shi-sta-ww}, which makes the following concepts well-defined. Since $\mathcal{H}(G)$ is real, symmetric and positive semi-definite, we deduce that its eigenvalues, called {\it Helmholtzian eigenvalues} (or {\it $\mathcal{H}$-eigenvalues}) of a graph $G$, are real and non-negative. Therefore, they can be listed as
\begin{equation*}\label{H-eigen1}
\lambda_1(G)\ge\lambda_2(G)\ge\cdots\ge\lambda_m(G) \ge 0,
\end{equation*}
where $m$ is the number of edges, of course. Then the {\it Helmholtzian spectrum} (or the \textit{$\mathcal{H}$-spectrum})  $\Sp_{\mathcal{H}}(G)$ of  $G$ consists of its $\mathcal{H}$-eigenvalues.
If $\lambda_1 > \lambda_2 > \cdots > \lambda_k$ are distinct $\mathcal{H}$-eigenvalues of $G$, then the  {\it $\mathcal{H}$-spectrum} of $G$ is written as
$$
{\Sp}_{\mathcal{H}}(G) = \left\{\!\!\begin{array}{cccc}
                                 \lambda_1(G) & \lambda_2(G) & \cdots & \lambda_k(G) \\
                                  m_1  &  m_2     & \cdots &  m_k
                                \end{array}
                         \!\!\right\},
$$
where $m_i$ is multiplicity of $\lambda_i$ ($1 \leq i \leq k$).

\section{Spectral Properties of Helmholtzian Matrix}\label{sec:main}

In this section, we sequentially investigate the spectral properties of Helmholtzian matrix. Precisely, in Subsection~\ref{H48} we consider the graphs with few $\mathcal{H}$-eigenvalues. Certain upper bounds for the least $\mathcal{H}$-eigenvalue are obtained in Section \ref{H46}. In Subsection~\ref{H47} we determine the nullity of a graph,  the multiplicity of zero in the $\mathcal{H}$-spectrum. The characteristic polynomial of $\mathcal{H}(G)$ is considered in Subsection \ref{H45}. Graph products and the resulting $\mathcal{H}$-spectrum are dealt with in Subsection~\ref{H49}. In particular, we establish the Helmholtzian matrix of the join of two graphs, and compute the corresponding spectrum in regular case. This result is used in Subsection \ref{H41} where we deal with graphs whose $\mathcal{H}$-spectrum consists entirely of integers.

\subsection{Graphs with Few $\mathcal{H}$-eigenvalues}\label{H48}

In this subsection, `a graph has $s$ eigenvalues', means that $G$ has exactly $s$ distinct eigenvalues. For the adjacency matrix, the problem of graphs with few eigenvalues originated in the study of  Doob \cite{doob}. In his view {\it few} is at most five, and he has determined a class of regular graphs with five eigenvalues related to Steiner triple systems. At present, this topic has been paid much attention, and there are many results employing association schemes, designs, codes and similar structures. Afterwards, van Dam  has made important contributions in his papers and Ph.D.~thesis \cite{dam-thesis}.  In the subsection, we consider  this topic for the Helmholtzian matrix.

Let $G$ be a graph of order $n$ and size $m$. Since $\mathcal{H}(G) = \mathcal{B}\mathcal{B}^{\intercal} + \mathcal{C}^{\intercal}\mathcal{C}$ is a
real and symmetric, there exist orthonormal eigenvectors $\mathbf {x_1},\mathbf{x_2},\ldots,\mathbf{x_m}$ forming a basis of $\mathbb{R}^m$. Therefore, if
$\lambda_1\ge\lambda_2\ge\cdots\ge\lambda_m$ are the eigenvalues, then
\[\mathcal{H}(G)=\lambda_1 \mathbf{x}_1\mathbf{x}_1^{\intercal}+\lambda_2\mathbf{x}_2\mathbf{x}_2^{\intercal}+\cdots+\lambda_m\mathbf{x}_m\mathbf{x}_m^{\intercal}.\]
Moreover, if $\lambda_1>\lambda_2>\cdots>\lambda_s$ are all distinct eigenvalues, then the {\it spectral decomposition} of $\mathcal{H}(G)$ is
\[\mathcal{H}(G)=\lambda_1\mathcal{P}_1+\lambda_2\mathcal{P}_2+\cdots+\lambda_s\mathcal{P}_s,\]
where $\mathcal{P}_i$ is the orthonormal projection of $\mathbb{R}^m$ onto the eigenspace
$\mathcal{E}_{\lambda_i}$ for $1\le i\le s$. It is well known that $\mathcal{P}_i$ is represented in the following form, but for the sake of completeness we include a short proof.

\begin{lem}\label{lem-4-26}
If $\lambda_1>\lambda_2> \cdots >\lambda_s$ are all distinct eigenvalues of $\mathcal{H}(G)$, then $\mathcal{P}_i=f_i(\mathcal{H}(G))$ where $f_i(x)=\frac{\prod_{j\ne i}(x-\lambda_j)}{\prod_{j\ne i}(\lambda_i-\lambda_j)}$.
\end{lem}

\begin{proof}
For convenience, denote  $\mathcal{H}=\mathcal{H}(G)$ and $F_i=f_i(\mathcal{H}(G))$. It suffices to show that $F_i\mathbf{x}\in\mathcal{E}_{\lambda_i}$ for every $\mathbf{x}\in\mathbb{R}^m$ and $F_i^2=F_i$, where $\mathcal{E}_{\lambda_i}$ is the eigenspace of $\lambda_i$. It is obvious that, for every $\mathbf{y}\in\mathcal{E}_{\lambda_j}$, $F_i\mathbf{y}=\mathbf{y}$ if $j=i$, and $F_i\mathbf{y}=\mathbf{0}$ otherwise. For every $\mathbf{x}\in\mathbb{R}^m$, let $\mathbf{x}=\sum_{j=1}^s\mathbf{x}_j$ where $\mathbf{x}_j\in\mathcal{E}_{\lambda_j}$. Therefore, we have
$F_i\mathbf{x}=F_i\mathbf{x}_i=\mathbf{x}_i\in \mathcal{E}_{\lambda_i}$. Moreover, $F_i^2\mathbf{x}=F_i\mathbf{x}_i=\mathbf{x}_i=F_i\mathbf{x}$, which yields $F_i^2=F_i$.
\end{proof}

From Lemma \ref{lem-4-26}, we deduce a characterization of graphs having  $s$  $\mathcal{H}$-eigenvalues.

\begin{prop}\label{cor-4-33}
For a graph $G$ of size $m$ with the Helmholtzian matrix $\mathcal{H}$,  the following statements are equivalent:
\begin{itemize}
    \item[(i)] $\mathcal{H}$ has $s$  $\mathcal{H}$-eigenvalues;
    \item[(ii)] The space $\{f(\mathcal{H})\mid f(x)\in \mathbb{R}[x]\}$ has dimension $s$;
    \item[(iii)] There exist real numbers $\lambda_1>\lambda_2>\cdots>\lambda_s$ such that $\mathcal{H}-\lambda_iI$ is singular, $f_i(\mathcal{H})^2=f_i(\mathcal{H})$ and $\mathcal{H}f_i(\mathcal{H})\mathbf{x}=f_i(\mathcal{H})\mathbf{x}$ for every $\mathbf{x}\in\mathbb{R}^m$ and $1\le i\le s$, where $f_i(x)=\frac{\prod_{j\ne i}(x-\lambda_j)}{\prod_{j\ne i}(\lambda_i-\lambda_j)}$.
\end{itemize}
\end{prop}

Recall that a connected graph $G$ has diameter at most $s-1$ if  $G$ has $s$ eigenvalues of the adjacency matrix. This fact is slightly different for the $\mathcal{H}$-eigenvalues.

\begin{thm}\label{lem-f-1}
Let $G$ be a connected graph with  $s$   $\mathcal{H}$-eigenvalues. Then the diameter of $G$ is at most $s$.
\end{thm}

\begin{proof}
Let $\mathcal{D}(G)$ be the diagonal matrix indexed by the edge set of $G$ with the $e$th diagonal entry $\Delta(e)+2$. Due to Theorem \ref{H-def}, we have $\mathcal{H}(G)= \mathcal{D}(G)-\mathcal{A}(G)$, where $\mathcal{D}(G)=\diag(\triangle(e)+2\mid e\in E(G))$ and $\mathcal{A}=(a_{ee'})$ is the square matrix indexed by the edge set of $G$ with
\begin{equation}\label{D+A}
a_{ee'}=\left\{\begin{array}{ll}1,&e\leftrightarrow e'\text{ and }e\not\vartriangle e'\\-1,&e'\overset{\pm}{\sim} e\text{ and }e\not\vartriangle e'\\0,& \text{otherwise}.\end{array}\right.
\end{equation}

Let $v_0$ and $v_d$ be a pair of vertices with $d(v_0,v_d)=d$, the diameter of $G$. Assume that $v_0v_1\ldots v_{d}$ is a path from $v_0$ to $v_d$, and let $e_i=v_{i-1}v_{i}$ for $1\le i\le d$. One may see $G$ as a rooted graph with root $v_0$, and denote  $V_i=\{v\in V(G)\mid d(v_0,v)=i\}$ for $0\le i\le d$. Since the $\mathcal{H}$-eigenvalues do not depend on the orientation of $G$, we may assume that $u\rightarrow v$ for any $uv\in E(G)$, with $u\in V_i$ and $v\in V_{i+1}$ for $0\le i\le d-1$.

We prove two claims.

{\it Claim 1}: $(\mathcal{H}^k(G))_{e_1e_{d}}=0$, for $k\le d-2$. In fact, if $k\le d-2$, for any edges $x_2,x_3,\ldots,x_{k}$, we have $a_{e_1x_2}a_{x_2x_3}\cdots a_{x_{k}e_{d}}=0$,
since otherwise $e_1x_2x_3\ldots x_{k}e_{d}$ forms a path from $v_0$ to $v_d$. This yields $d(v_0,v_d)\le k+1\le d-1$, which is a contradiction. Therefore,
$(A^k)_{e_1e_{d}}=\sum_{x_2,x_3,\ldots,x_{k}}a_{e_1x_2}a_{x_2x_3}\cdots a_{x_{k}e_{d}}=0$, which implies $(H^k)_{e_1e_{d}}=0$.

{\it Claim 2}: $(\mathcal{H}^{d-1}(G))_{e_1e_{d}}>0$. Actually, a direct computation shows that
\[({\mathcal{A}}(G)^{d-1})_{e_1e_{d}}=\sum_{x_2,\ldots,x_{d-1}}a_{e_1x_2}a_{x_2x_3}\cdots a_{x_{d-1}e_{d}}.\]
If  $a_{e_1x_2}a_{x_2x_3}\cdots a_{x_{d-1}e_{d}}\ne 0$, then $e_1x_2x_3\ldots
x_{d-1}e_{d}$ forms a path from $v_0$ to $v_d$. Assume that $x_i=u_{i-1}u_{i}$ for $2\le i\le
d-1$, where $u_1=v_1$ and $u_{d-1}=v_{d-1}$. From $d(v_0,v_{d})=d$  we have
$d(v_{0},u_i)=i$, and then $u_i\in V_i$ for $1\le i\le d-1$, which leads to
$a_{e_1x_2}a_{x_2x_3}\cdots a_{x_{d-1}e_{d}}=1$. Hence,
$({\mathcal{A}}^{d-1}(G))_{e_1e_{d}}\ge 0$. Moreover, since $a_{e_1e_2}a_{e_2e_3}\cdots
a_{e_{d-1}e_{d}}=1$, it holds $(A^{d-1})_{e_1e_{d}}>0$. Thereby, $(H^{d-1})_{e_1e_{d}}>0$.

We now  return to the main course. Suppose  that $d\ge s+1$. Since $\mathcal{H}(G)$ is symmetric, the degree of its minimal polynomial is $s$. From $s \le d-1$, we obtain
\[\mathcal{H}^{d-1}(G)=c_{m-1}\mathcal{H}^{m-1}(G)+c_{m-2}\mathcal{H}^{m-2}(G)+\cdots+
c_1\mathcal{H}(G)+c_0I,\]
where $c_0,c_1,\ldots,c_{m-1}\in \mathbb{R}$. However, Claims 1 and 2 give $(\mathcal{H}^{d-1}(G))_{e_1e_{d}}>0$; whereas  $(\mathcal{H}^k(G))_{e_1e_{d}}=0$ for $0\le k\le m-1\le d-2$, which is a contradiction and the proof is completed.
\end{proof}

Easily to verify that the $\mathcal{H}$-eigenvalues of the complete graph $K_n$ are $n$'s with multiplicity ${n \choose 2}$. Conversely, Theorem \ref{lem-f-1} indicates that the diameter of $G$ is at most 1 if $G$ has one $\mathcal{H}$-eigenvalue, and thus $G$ is complete. Hence, the following result follows.

\begin{prop}\label{cor-f-1}
A connected graph $G$ has one distinct $\mathcal{H}$-eigenvalue if and only if $G \cong K_n$ ($n \geq 2$).
\end{prop}

Note that, in case of either the adjacency matrix or the  Laplacian matrix,  a graph has the previous spectral property if and only if  $G \cong K_1$, and $G$ has two  eigenvalue if and only if $G \cong K_n$ ($n \geq 2$).

Naturally, the next step is to characterize  graphs with two  $\mathcal{H}$-eigenvalues. We consider the $\mathcal{H}$-spectrum of the so-called {\it generalized windmill graph},  defined as $K_{n_0} \vee (K_{n_1}\cup K_{n_2}\cup\cdots\cup K_{n_k})$. The following lemma is needed.

\begin{lem}[{\cite{gant}}]\label{holed-matrix}
	Let $A$ and $D$ be invertible square matrices and
	$M=\begin{pmatrix}
		A & B \\
		C & D
	\end{pmatrix}
	$.
	Then $$\det(M)=\det(A)\det(D-CA^{-1}B)=\det(D)\det(A-BD^{-1}C).$$
\end{lem}

\begin{lem}\label{sub-prop}
	If $n_1,n_2,\ldots,n_k$ are positive integers with $\sum_{i=1}^kn_i=n$ and
	$$
	{
		N=\begin{pmatrix}
			(n_1+1)I_{n_1} & J& \cdots & J \\
			J & (n_2+1)I_{n_2}& \cdots & J \\
			\vdots & \vdots & \ddots & \vdots \\
			J & J & \cdots & (n_s+1)I_{n_k}
	\end{pmatrix}},
	$$
	then
	\begin{equation}\label{spN}
		{\Sp}(N) =
		\left\{\!\!\begin{array}{ccccc}
			n+1  & n_1+1  & \cdots & n_k+1 & 1  \\
			1    & n_1-1 & \cdots & n_k-1 & k-1
		\end{array}
		\!\!\right\}.
	\end{equation}
\end{lem}

\begin{proof}
	The characteristic polynomial of $N$ is $\phi(N,\lambda) = \det(\lambda I-N) = \det(N_1-N_2)$, where
	\begin{equation*}\label{m-s2}
		N_1=\diag(N_{11},N_{12},\ldots,N_{1k})
		\;\; {\mbox{and}} \;\;
		N_2=J_{n\times n}=
		\mathbf{1}_n\mathbf{1}_n^{\intercal},
	\end{equation*}
	with $N_{1i}=(\lambda-n_i-1)I_{n_i}-J$ for $1\le i\le k$.
	Define
	$
	Z=
	\begin{pmatrix}
		N_1 & \mathbf{1}_{n} \\
		\mathbf{1}_n^{\intercal} & 1
	\end{pmatrix}
	$.
	Then, by Lemma \ref{holed-matrix}, we deduce that $$\det(Z) = \det(N_1)(1 - \mathbf{1}^{\intercal}N_1^{-1}\mathbf{1}) = \det(N_1-\mathbf{1}_n\mathbf{1}^{\intercal}_n) = \det(N_1-N_2) = \phi(N,\lambda).$$
	Since $\det(N_{1i})=(\lambda-1)(\lambda-n_i-1)^{(n_i-1)}$, we have $\det(N_1)=(\lambda-1)^k\prod\limits_{i=1}^k(\lambda-n_i-1)^{n_i-1}$. Note that $N_{1i}\mathbf{1}=(\lambda-1)\mathbf{1}$. It follows $N_{1i}^{-1}\mathbf{1}=\frac{1}{\lambda-1}\mathbf{1}$, and thus $\mathbf{1}^{\intercal}N_1^{-1}\mathbf{1}=\sum_{i=1}^k\mathbf{1}^{\intercal}N_{1i}^{-1}\mathbf{1}=\sum_{i=1}^k\frac{n_i}{\lambda-1}=\frac{n}{\lambda-1}$.
	Hence, $$\phi(N,\lambda)=\Big((\lambda-1)^k\prod\limits_{i=1}^k(\lambda-n_i-1)^{n_i-1}\Big)\Big(1-\frac{n}{\lambda-1}\Big)=(\lambda-1)^{k-1}(\lambda-1-n)\prod\limits_{i=1}^k(\lambda-n_i-1)^{n_i-1},$$ which leads to \eqref{spN}.
\end{proof}

We apply the previous result to compute the $\mathcal{H}$-spectrum of a generalized windmill.

\begin{prop}\label{H-gwg1}
	Let $G\cong K_{n_0} \vee(K_{n_1}\cup K_{n_2}\cup\cdots\cup K_{n_k})$ be the generalized windmill graph of order $n=\sum_{i=n_0}^{n_k}$. Then
	\begin{equation}\label{H-gwg}
		{\Sp}_{\mathcal{H}}(G)
		=\left\{\!\!\begin{array}{ccccc}
			n                 & n_0       & n_0 + n_1                   & \cdots & n_0+n_k \\
			\frac{n_0(n_0 +1)}{2}  & n_0(k-1)  &  \frac{(2n_0+n_1)(n_1-1)}{2} & \cdots & \frac{(2n_0+n_k)(n_k-1)}{2}
		\end{array}
		\!\!\right\}.
	\end{equation}
\end{prop}

\begin{proof}
	For $G$, we assign an orientation in such a way  that $u\rightarrow v$ holds for every $u \in \bigcup_{i=1}^kV(K_{n_i})$ and $v \in V(K_{n_0})$. Denote  $E_1=\bigcup_{i=1}^k E(K_{n_i})$, $E_3=E(K_{n_0})$ and $E_{2i}=\{\{u,v_i\}\mid v_i \in V(K_{n_0}), u \in \bigcup_{j=1}^kV(K_{n_j})\}$, for $1\le i\le n_0$. Then	
		\begin{equation*}
			\mathcal{H}(G) =
			\begin{pmatrix}
				H_{1} &  &  &    &  \\
				& H_{21} &  &    &  \\
				&  & \ddots &    &  \\
				&  &  & H_{2n_0}   &  \\
				&  &  &    & H_{3}
			\end{pmatrix},
		\end{equation*}
	where, for $1\le i\le n_0$, $H_3=nI_{n_0\choose 2}$,
	{
		$$H_{1}=\begin{pmatrix}
			(n_1+n_0)I_{n_1\choose 2} &  & \\
			& \ddots &  \\
			&  & (n_k+n_0)I_{n_k\choose 2}
		\end{pmatrix}
		$$
		and
		$$
		H_{2i}=
		\begin{pmatrix}
			(n_1+n_0)I_{n_1} & J & \cdots & J \\
			J & (n_2+n_0)I_{n_2} & \cdots & J \\
			\vdots & \vdots & \ddots & \vdots \\
			J & J & \cdots & (n_k+n_0)I_{n_k}
		\end{pmatrix}.
		$$}
By employing Lemma~\ref{sub-prop}, we arrive at \eqref{H-gwg}.
\end{proof}

Especially, the graph $K_t\vee(sK_1)$ is called a \textit{complete split graph}, the  $\mathcal{H}$-spectrum of which is obtained on the basis of Proposition~\ref{H-gwg1}.

\begin{cor}\label{CSG}
For $s \geq 2$,
$$
{\Sp}_{\mathcal{H}}(K_t\vee(sK_1)) = \left\{\!\!\begin{array}{cc}
                                 s+t & t  \\
                       t+1\choose 2  &  (s-1)t
                                \end{array}
                         \!\!\right\}.
$$
\end{cor}

Actually, the complete split graphs belong to the class of the  {\it cographs}, i.e., $P_4$-free graphs. Cographs have appeared during  1960s due to their wide applications in areas like parallel computing \cite{nak-cograph} or even biology \cite{gag-cograph}. So far, various characterizations of cographs have been presented \cite{bra-cogtraph}, and one of them states that a connected cograph it is the join of two cographs. We will use this in the forthcoming  proofs.

\begin{thm}
A connected graph $G$ has  two  $\mathcal{H}$-eigenvalues if and only if $G\cong K_t\vee(sK_1)$, with $s \geq 2$.
\end{thm}

\begin{proof}
The sufficiency follows from Corollary \ref{CSG}. We next show the necessity. Since $G$ has two $\mathcal{H}$-eigenvalues, by Proposition \ref{cor-4-33}(ii), there exist $a,b\in\mathbb{R}$ such that
\begin{equation}\label{h2hI}
{\mathcal{H}}^2(G)=a\mathcal{H}(G)+bI.
\end{equation}

{\it Claim 1}: $G$ is $C_4$-free. Suppose for a contradiction that $G$ contains an induced cycle $C_4=v_1v_2v_3v_4v_1$. Without loss of generality, assume that $v_1\rightarrow v_2\rightarrow v_3\rightarrow v_4\rightarrow v_1$. Set $e=v_1v_2$ and $e'=v_3v_4$. A direct  calculation shows that $(\mathcal{H}^2(G))_{ee'}=2$. However, $\mathcal{H}(G)_{ee'}=0$, which leads to the impossible scenario $2=0$.

With a similar method,  $G$ is also $C_5$-free, and then $G$ is  $P_4$-free, i.e., $G$ is a cograph.

{\it Claim 2}: $G$ is $K_1\vee (K_2\cup K_1)$-free. Again, suppose  that $\{v_1,v_2,v_3,v_4\}$ induces  $K_1\vee(K_2\cup K_1)$ with $v_1$ being the pendant vertex and $v_2$ being the vertex of degree $3$ in this subgraph, respectively. Assume that $v_1\rightarrow v_2\rightarrow v_3\rightarrow v_4\rightarrow v_2$, and set $e=v_2v_3$ and $e'=v_2v_4$. Clearly, $\mathcal{H}(G)_{ee'}=0$. On the other hand, if $\mathcal{H}(G)_{ex}\mathcal{H}(G)_{xe'}\ne 0$, then $x=v_2u$ for some $u\not\in \{v_3,v_4\}$. If $u\rightarrow v_2$, then $\mathcal{H}(G)_{ex}\mathcal{H}(G)_{xe'}=(-1)\cdot 1=-1$. If $v_2\rightarrow u$, then $\mathcal{H}(G)_{ex}\mathcal{H}(G)_{xe'}=1\cdot(-1)=-1$. Therefore, we have
$(\mathcal{H}^2(G))_{ee'}=\sum_{x}\mathcal{H}(G)_{ex}\mathcal{H}(G)_{xe'}<0$, which contradicts  \eqref{h2hI}.

By the previous claims, $G$ is $\{P_4,C_4,K_1\vee (K_2\cup K_1)\}$-free. Since $G$ is $P_4$-free, it is a join $G\cong G_1 \vee G_2$. Without loss of generality, we may assume that no vertex in $V(G_1)$ is adjacent all other vertices in $G_1$, as for otherwise this vertex may be moved to $G_2$. By Proposition \ref{cor-f-1}, $G$ is not a  complete graph. Note that $G$ contains  $2K_1$ as an induced subgraph, and without loss of generality we may suppose that $G_1$ contains  $2K_1$. Since $G$ is $C_4$-free, $G_2$ is $2K_1$-free, and hence $G_2$ is a complete graph.

We next show that $G_1$ is edgeless. Assume to the contrary that $G_1$ has an edge. Since $G_1$ contains $2K_1$, it contains either  $K_2\cup K_1$ or $P_3$ as an induced subgraph. In the former case, $G$ contains  $K_1\vee(K_2\cup K_1)$, which is impossible. If the latter happens, we set $P_3=uvw$ and claim that  any other vertex in $G_1$ is adjacent to $v$. Otherwise, there exists $v'\in V(G')$ with $v'\not\sim v$. If $v'\not\sim u$ or $v'\not\sim w$, then $G_1$ contains  $K_2\cup K_1$, and thus $G$ contains  $K_1\vee (K_2\cup K_1)$, which is impossible. If $v'\sim u$ and $ v' \sim w$, then $\{u,v,w,v'\}$ forms  $C_4$, which is the final contradiction.
\end{proof}

We conclude this subsection by the next natural problem.

\begin{problem}
Characterize  graphs with three  $\mathcal{H}$-eigenvalues.
\end{problem}

\subsection{The Least $\mathcal{H}$-eigenvalue}\label{H46}

For a non-zero vector $\mathbf{x}\in \mathbb{R}^m$,  $\frac{\mathbf{x}^{\intercal}\mathcal{H}(G)\mathbf{x}}{\mathbf{x}^{\intercal}\mathbf{x}}$ is called the {\it Rayleigh quotient} for $\mathcal{H}(G)$ and $\mathbf{x}$. It is well known that
\[\lambda_1=\max\frac{\mathbf{x}^{\intercal}\mathcal{H}(G)\mathbf{x}}{\mathbf{x}^{\intercal}\mathbf{x}}\,\text{ and }\,\lambda_m=\min\frac{\mathbf{x}^{\intercal}\mathcal{H}(G)\mathbf{x}}{\mathbf{x}^{\intercal}\mathbf{x}}.\]

For a vertex $v$ and an edge $e$ with $v\in \mathcal{V}(e)$, let ${\sgn}(v,e)=-1$ if $v=e^-$, and ${\sgn}(v,e)=1$
otherwise. For an edge $e$ and a triangle $\vartriangle$ with $e\in\vartriangle$, let ${\sgn}(e,\vartriangle)=1$ if  orientations of $e$ and $\vartriangle$ agree, and
${\sgn}(v,e)=-1$ otherwise.
Since $(\mathcal{C}\mathbf{x})_{\vartriangle}=\sum_{e\in\vartriangle}{\sgn}(e,\vartriangle)x_e$, we
have
\begin{equation}\label{eq-1}
\frac{1}{|T|}\big(\sum_{\vartriangle}\sum_{e\in\vartriangle}{\sgn}(e,\vartriangle)x_e\big)^2\le
\mathbf{x}^{\intercal}\mathcal{C}^{\intercal}\mathcal{C}\mathbf{x}=\sum_{\vartriangle}\big(\sum_{e\in\vartriangle}{\sgn}(e,\vartriangle)x_e\big)^2\le\sum_{\vartriangle}3\sum_{e\in\vartriangle}x_e^2=\sum_{e\in
E(G)}3\Delta(e)x_e^2.
\end{equation}
Similarly, from $(\mathcal{B}^{\intercal}\mathbf{x})_v=\sum_{e\sim v}{\sgn}(v,e)\mathbf{x}_e$, we deduce

\begin{align}\label{eq-2}
\frac{1}{|V|}\big(\sum_{v}\sum_{e\sim v}{\sgn}(v,e)x_e\big)^2 \nonumber
&\le \mathbf{x}^{\intercal}\mathcal{B}\mathcal{B}^{\intercal}\mathbf{x}=\sum_{v}\big(\sum_{e\sim v}{\sgn}(v,e)x_e\big)^2\\
&\le\sum_vd(v)\sum_{e\sim v}x_e^2=\sum_{e\in E(G)}\big(d(e^+)+d(e^-)\big)x_e^2.
\end{align}
On the other hand, we have
\begin{equation}\label{eq-3}
\begin{array}{lll}
\mathbf{x}^{\intercal}\mathcal{H}(G)\mathbf{x}&=&\sum_{e}(\Delta(e)+2)x_e^2-2\sum_{e\leftrightarrow
e',e\not\vartriangle e'}x_ex_e'+2\sum_{e\overset{\pm}{\sim} e',e\not\vartriangle
e'}x_ex_e'\\[2mm]
&=&\sum_{e}\big(3\Delta(e)+4-d(e^+)-d(e^-)\big)x_e^2+\sum_{e\leftrightarrow e',e\not\vartriangle
e'}(x_e-x_e')^2\\
&\phantom{+}&+\sum_{e\overset{\pm}{\sim} e',e\not\vartriangle e'}(x_e+x_e')^2.
\end{array}
\end{equation}

From our experience, the inequalities \eqref{eq-1} and \eqref{eq-2} as well as the equality
\eqref{eq-3} have a high potential  in determining  upper and lower bounds for
$\lambda_1$ and the related topics \cite{3w-2b}. Here we provide two upper bounds for the least $\mathcal{H}$-eigenvalue $\lambda_{m}$.

\begin{thm}
Let $G$ be a connected graph of order $n$ whose least $\mathcal{H}$-eigenvalue is $\lambda_m$. Then
\begin{itemize}
\item[{(i)}]
$\lambda_{m} \le \min\{\Delta(e)+2 \mid e \in E(G)\}$ with equality if and only if $G=K_n$;\\[-5mm]
\item[{(ii)}]
If $G \ncong K_n$, then
$\lambda_m\le\min\big\{\frac{\Delta(e_i)+\Delta(e_j)}{2}+1\big\}$ where $e_i$ and $e_j$ form an induced $P_3$.
\end{itemize}
\end{thm}

\begin{proof}
(i): For an edge $e$, it is clear that $\lambda_{\min} \le \frac{\alpha_e^{\intercal}H(G)\alpha_e}{\alpha_e^{\intercal}\alpha_e}=\Delta_e+2$, where $\alpha_{e}$ represents the eigenvector with the $e$th component being 1 and the rest being 0, and $\lambda_{\min}=\min \Delta(e)+2$ when $G\cong K_n$.

Conversely, assume that $\Delta(e_0)=\min \Delta(e)$ and $\lambda_{\min}=\Delta(e_0)+2$.  For an edge $e_1$ with $e_1\sim e_0$, we have $e_1\triangle e_0$. Otherwise, we face a contradiction by considering the $e_1$th entry of both sides of $H(G)\delta_{e_0}=\lambda\delta_{e_0}$. For a triangle $\{e_1,e_2,e_3\}$ containing $e_1$, exactly one  $i\in\{2,3\}$ satisfies $e_i\sim e_0$, and thus $e_i\triangle e_0$. This indicates that $\Delta(e_1)\le \Delta(e_0)$, and thus $\Delta(e_1)=\Delta(e_0)=\min\Delta(e)$. By regarding $e_1$ as $e_0$ and repeating the procedure, we conclude that $G\cong K_n$ since $G$ is connected. Thus, (i) follows.

(ii): For every pair of edges  $e_i$ and $e_j$ forming an induced $P_3$, we have either $e_i\leftrightarrow e_j$ or $e_i\overset{\pm}{\sim} e_j$. Let $\mathbf{x}$ be the vector indexed by the edges with $x_{e_i}=x_{e_j}=1$ and $x_e=0$ for $e\not\in\{e_i,e_j\}$ when $e_i\leftrightarrow e_j$, and  $x_{e_i}=-x_{e_j}=1$, $x_{e}=0$ for $e\not\in\{e_i,e_j\}$ when $e_i\overset{\pm}{\sim} e_j$. By a direct computation, we arrive at $$\lambda_m\le \frac{x^{\intercal}H(G)x}{x^{\intercal}x}=\frac{\Delta(e_i)+\Delta(e_j)}{2}+1,$$ and  (ii) follows.
\end{proof}

The second least Laplacian eigenvalue, also known as the algebraic connectivity, plays a significant role as it reveals many structural properties of a graph~\cite{fie-CMJ}. In the case of signed graphs, this role is attributed to the least Laplacian eigenvalue~\cite{Bel}. Accordingly, we conclude this part with the following problem.

\begin{problem}
Find relationships between the least or the second least $\mathcal{H}$-eigenvalue of a graph $G$ and its structural properties.
\end{problem}

\subsection{$\mathcal{H}$-nullity}\label{H47}

The multiplicity of the zero eigenvalue in the spectrum of a graph is the so-called \textit{nullity}. Denoted by  $\eta_M(G)$ the {\it nullity} of a graph $G$ with respect to a graph matrix $M(G)$. For the adjacency matrix $A(G)$, Collatz and Sinogowitz \cite{col-sin} have posed the problem of characterizing all graphs with $\eta_A(G) >0$. This question has strong chemical background, because $\eta_A(G) = 0$ is a necessary condition for a so-called conjugated molecule to be chemically stable, where $G$ is the graph representing the carbon-atom skeleton of this molecule. For the Laplacian matrix $L(G)$, it is well-known that the nullity of a graph is exactly the number of its connected components.

We proceed with $\mathcal{H}$-nullity and estimate it in terms of the order, the size and the number of triangles of $G$.

\begin{lem}\label{prop:zero}
Let $G$ be a graph with $m$ edges. Zero appears in the $\mathcal{H}$-spectrum of $G$ if and only if there exists a non-zero eigenvector $\mathbf{x}$ such that $\mathcal{B}^\intercal\mathbf{x}=\mathbf{0}$ and $\mathcal{C}\mathbf{x}=\mathbf{0}$.
\end{lem}

\begin{proof}Suppose that 0  is the $\mathcal{H}$-spectrum of $G$, and let $\mathbf{x}$ be a corresponding eigenvector. Then $0=\mathbf{x}^\intercal (\mathcal{B}\mathcal{B}^\intercal+\mathcal{C}^\intercal\mathcal{C})\mathbf{x}=\mathbf{x}^\intercal (\mathcal{B}\mathcal{B}^\intercal)\mathbf{x}+\mathbf{x}^\intercal(\mathcal{C}^\intercal\mathcal{C})\mathbf{x}$. Since both  $\mathcal{B}\mathcal{B}^\intercal$ and $\mathcal{C}^\intercal\mathcal{C}$ are positive semi-definite,
	both summands on the right hand side are $\geq 0$. Since they are summing up to 0, they are equal to 0, which means that $\mathbf{x}$ is an eigenvector to 0 for both $\mathcal{B}\mathcal{B}^\intercal$ and $\mathcal{C}^\intercal\mathcal{C}$. For the former matrix, this yields $(\mathcal{B}^\intercal \mathbf{x})^\intercal \mathcal{B}^\intercal \mathbf{x}=0$, giving   $\mathcal{B}^\intercal \mathbf{x}=\mathbf{0}$. The identity $\mathcal{C} \mathbf{x}$ is obtained in a similar way.
	
	The opposite implication follows directly as $\mathcal{B}^\intercal\mathbf{x}=\mathcal{C}\mathbf{x}=\mathbf{0}$ yields $(\mathcal{B}\mathcal{B}^\intercal+\mathcal{C}^\intercal\mathcal{C})\mathbf{x}=\mathbf{0}$.
\end{proof}

\begin{lem}\label{lem-a-1}
Let $G$ be a graph with size $m(G)$. Then
\[\eta_{\mathcal{H}}(G)=m(G)-{\rm rank}\left(\begin{array}{c}\mathcal{B}(G)^{\intercal}\\\mathcal{C}(G)\end{array}\right),\]
where $\mathcal{B}(G)$ and $\mathcal{C}(G)$ are defined in \eqref{B-C}.
\end{lem}

\begin{proof}
Given an arbitrary orientation to $E(G)$ and $T(G)$, by Lemma \ref{prop:zero}  we get  $$\mathcal{H}(G){\rm \bf x}=(\mathcal{B}(G)\mathcal{B}(G)^{\intercal}+\mathcal{C}(G)^{\intercal}\mathcal{C}(G)){\rm \bf x}=\mathbf{0}$$ holds if and only if $\mathcal{C}(G){\rm \bf x}=\mathbf{0}$ and $\mathcal{B}(G)^{\intercal}{\rm \bf x}=\mathbf{0}$, which is equivalent to $$\left(\begin{array}{c}B(G)^{\intercal}\\\mathcal{C}(G)\end{array}\right){\rm \bf x}=\mathbf{0}.$$ Thereby, $\eta_{\mathcal{H}}(G)=m(G)-{\rm rank}\left(\begin{array}{c}\mathcal{B}(G)^{\intercal}\\\mathcal{C}(G)\end{array}\right)$.
\end{proof}

Note that ${\rm rank}(\mathcal{B}^{\intercal}(G))={\rm rank}(\mathcal{B}(G))$ and ${\rm rank}(\mathcal{C}(G))=t_G(\triangle)$, the number of triangles in $G$. Thus, we immediately get the following result.

\begin{cor}\label{cor-a-1}
Let $G$ be a graph with arbitrary orientations on $E(G)$ and $T(G)$. Then
\[\eta_{\mathcal{H}}(G)=m(G)-t_G(\triangle)-{\rm rank}(\mathcal{B}(G)).\]
\end{cor}

We proceed with a lemma.

\begin{lem}\label{lem-3}
If $G$ has a pendent vertex $v$ and $G'=G-v$, then $\eta_{\mathcal{H}}(G)=\eta_{\mathcal{H}}(G')$.
\end{lem}

\begin{proof}
Given arbitrary orientations to $E(G)$ and $T(G)$, we deduce $m(G')=m(G)-1$, ${\rm rank}(\mathcal{B}(G'))={\rm rank}(\mathcal{B}(G))-1$ and $t_{G'}(\triangle)=t_G(\triangle)$. Hence, the result follows from Corollary \ref{cor-a-1}.
\end{proof}

For a vertex $v$ of a graph $G$, let $N_G(v)$ denote the set of the neighbours of $v$. If $e=uv$ such that $N_G(u)\cap N_G(v) = \emptyset$, the {\it contraction} of $e$ is the replacement of $u$ and $v$ with a single vertex whose incident edges are the edges other than $e$ that are incident to $u$ or $v$, and the resulting graph is denoted by $G/e$.

\begin{lem}\label{lem-4}
Let $G$ be a graph with edge $e=uv$ such that $N(u)\cap N(v)=\emptyset$. If $G'=G/e$, then $$\eta_{\mathcal{H}}(G)=\eta_{\mathcal{H}}(G')+\triangle_{G'}(e).$$
\end{lem}

\begin{proof}
We assign orientations to $E(G)$ and $T(G)$ in such a way that $u$  and $v$ are respectively the tail and  the head of every edge incident with them. Clearly, $m(G') = m(G)-1$ and $t_{G'}(\triangle)=t_G(\triangle)-\triangle_G(e)$. It suffices to show that ${\rm rank}(\mathcal{B}(G))={\rm rank}(\mathcal{B}(G'))+1$ by Corollary \ref{cor-a-1}. For any vertex $z$, let $E_z=\{e\mid z\in \mathcal{V}(e)\}$. Denote  $E_1=E(G)\setminus(E_u\cup E_v)$, $E_2=E_u\setminus\{e\}$, $E_3=E_v\setminus\{e\}$, and let $u'$ be the vertex in $G'$ that replaces $u$ and $v$ in $G$. Then
\[
\mathcal{B}(G)^{\intercal}=\!\!\!\!\!\!\!\begin{array}{cc}
&\begin{array}{cccc}
E_1&E_2&E_3&\{e\}
\end{array}\\
\begin{array}{c}
u\\v\\ V\setminus\{u,v\}
\end{array}&
\left(\begin{array}{cccc}
\mathbf{0}&-\mathbf{1}^{\intercal}&\mathbf{0}&-1\\
\mathbf{0}&\mathbf{0}&\mathbf{1}^{\intercal}&1\\
B_1&B_2&B_3&\mathbf{0}
\end{array}\right)
\end{array} \;\; {\mbox{and}} \;\;
\mathcal{B}(G')^{\intercal}=\!\!\!\!\!\!\begin{array}{cc}
&\begin{array}{ccc}
E_1&E_2&E_3
\end{array}\\
\begin{array}{c}
u'\\ V'\setminus\{u'\}
\end{array}&
\left(\begin{array}{ccc}
\mathbf{0}&-\mathbf{1}^{\intercal}&\mathbf{1}\\
B_1&B_2&B_3
\end{array}\right)
\end{array}.
\]

If the $u'$th row of $B(G')^{\intercal}$ can be represented as a linear combination of  the other rows, then the same holds for the sum of the $u$th row and the $v$th row of $B(G)^{\intercal}$, and thus $${\rm rank}(\mathcal{B}(G'))={\rm rank}(B_1\; B_2\; B_3) \;\; {\mbox{and}} \;\; {\rm rank}(\mathcal{B}(G))={\rm rank}(B_1\; B_2\; B_3)+1.$$ Hence, ${\rm rank}(\mathcal{B}(G))= {\rm rank}(\mathcal{B}(G'))+1$. If the $u'$th row cannot be represented by the other rows, then $${\rm rank}(\mathcal{B}(G'))= {\rm rank}(B_1\; B_2\; B_3)+1\;\; {\mbox{and}} \;\; {\rm rank}(\mathcal{B}(G))= {\rm rank}(B_1\; B_2\; B_3)+2,$$ and we still have ${\rm rank}(\mathcal{B}(G))= {\rm rank}(\mathcal{B}(G'))+1$.
\end{proof}

The following result follows from the previous lemma.

\begin{cor}\label{cor-1}
Let $G$ be a graph with a cut-edge $e$. Then $\eta_{\mathcal{H}}(G)=\eta_{\mathcal{H}}(G/e)$.
\end{cor}

The next result follows from the fact that each edge of a tree is a cut-egde and Corollary \ref{cor-1}.

\begin{cor}\label{T-ni}
Let $\mathcal{T}$ be a tree. Then, $\eta_{\mathcal{H}}(\mathcal{T})=0$, and so $\mathcal{H}(\mathcal{T})$ is an invertible matrix.
\end{cor}

The main result in this subsection reads as follows.

\begin{thm}\label{thm-1}
Let $G$ be a graph with $w(G)$ components. Then $$\eta_{\mathcal{H}}(G)=m(G)-n(G)-t_G(\triangle)+w(G).$$
\end{thm}

\begin{proof}
For $1 \leq i \leq w(G)$, we consider each component $G_i$ of $G$ by induction on $t_{G_i}(\triangle)$.  If $t_{G_i}(\triangle)=0$, then ${\rm rank}(\mathcal{B}(G_i))={\rm rank}(\mathcal{B}(G_i)^{\intercal})={\rm rank}(L(G_i))=n_i-1$, because $G_i$ is connected. Therefore, $\eta_{\mathcal{H}}(G_i)=m(G_i)-n(G_i)+1$. Assume that the statement is true for $t_{G_i}(\triangle) \le t$. Assume that $t_{G_i}(\triangle)=t+1$. Taking an edge $e$ belonging to some triangle, let $\tilde{G}_i$ be the graph obtained from $G_i$ by replacing $e$ with a directed path $e^-\rightarrow u\rightarrow e^+$. Clearly, $G_i=\tilde{G_i}/\{e^-,u\}$ and $N_{\tilde{G_i}}(e^-)\cap N_{\tilde{G_i}}(u)=\emptyset$. Together with Lemma \ref{lem-4}, this yields
\begin{equation}\label{eta1}
\eta_{\mathcal{H}}(\tilde{G}_i)=\eta_{\mathcal{H}}(G_i)+\triangle_{G_i}(e).
\end{equation}
By the induction hypothesis,
\begin{equation}\label{eta2}
\eta_{\mathcal{H}}(\tilde{G}_i)=m(\tilde{G}_i)-n(\tilde{G}_i)-t(\tilde{G}_i)+1.
\end{equation}
By substituting $m(\tilde{G}_i)=m(G_i)+1$, $n(\tilde{G}_i)=n(G_i)+1$ and $t_{\tilde{G}_i}(\triangle)=t_{G_i}(\triangle)-\triangle_{G_i}(e)$ into \eqref{eta2} and employing \eqref{eta1}, we arrive at
\[(m(G_i)+1)-(n(G_i)+1)-(t_{G_i}(\triangle)-\triangle_{G_i}(e))+1=\eta_{\mathcal{H}}(\tilde{G}_i)
=\eta_{\mathcal{H}}(G_i)+\Delta_{G_i}(e),\]
resulting in $\eta_{\mathcal{H}}(G_i)=m(G_i)-n(G_i)-t_{G_i}(\triangle)+1$. Consequently,
$$\eta_{\mathcal{H}}(G)=\sum_{i=1}^{w(G)}\eta_{\mathcal{H}}(G_i)
=\sum_{i=1}^{w(G)}[m(G_i)-n(G_i)-t_{G_i}(\triangle)+1]=m(G)-n(G)-t_G(\triangle)+w(G),$$
which completes the proof.
\end{proof}

We conclude this  subsection with the following observation; a consequence of the previous theorem.

\begin{cor}
Under the assumptions of Theorem \ref{thm-1}, $t_G(\triangle) =  m(G)-n(G)-\eta_{\mathcal{H}}(G)+w(G)$.
\end{cor}

The number of triangles in a graph or a network is a significant metric invariant recognized in an extensive  range of  applications, see \cite[eg.]{has-dav,pan-etal,rez-etal,zhang-etal}. Accordingly, computing the number of triangles (in terms of other graph invariants) is a problem posed in the GraphChallenge competition~\cite{kep}, and the previous proposition can be seen as a contribution in this domain.

Regarding Corollary \ref{T-ni}, we  pose the following problem to finish this subsection.

\begin{problem}
For a tree $T$, determine the  inverse matrix $\mathcal{H}^{-1}(T)$.
\end{problem}

\subsection{$\mathcal{H}$-polynomial}\label{H45}

We denote the characteristic polynomial of $\mathcal{H}(G)$ by $$\phi_{\mathcal{H}}(G,\lambda) = \det(\lambda I_m - \mathcal{H}(G))=\lambda^m + c_1\lambda^{m-1}+\cdots+c_{m-1}\lambda+c_m,$$ and abbreviate it to the \textit{$\mathcal{H}$-polynomial} of a graph $G$. In this section our focus is on the combinatorial interpretations of the coefficients $c_i$.

Note that we used the matrix \eqref{D+A} to show Lemma \ref{lem-f-1}. In fact, it contains more profound graph-theoretic information and has been employed to investigate the irreducibility of the Helmholtzian matrix \cite{lu-shi-sta-ww}. At this event we need to introduce the signed graphs. A {\it signed graph} is an ordered pair $(G,\sigma)$, where $G$ is an ordinary graph, called the underlying graph, and $\sigma\colon E(G)\longrightarrow\{+1,-1\}$ is the signature on the edge set $E(G)$. The adjacency matrix  $A(\dot{G})$  of $\dot{G}$ is obtained from the adjacency matrix $A(G)$ by reversing the sign of all 1s which correspond to negative edges. For a graph $G$ with an
orientation on $E(G)$, let $\Lambda(G)$ denotes the following signed graph with loops:
\begin{enumerate}
\item[(a)] the vertex set is $V=E(G)$;
\item[(b)] there are $\triangle(e)+2$ loops with sign $1$ attached at every vertex $e$;
\item[(c)] the set of edges with sign $1$ is $E^+=\{\{e,e'\}\mid e\overset{\pm}{\sim} e'
    \text{ and }e\not\vartriangle e'\}$;
\item[(d)] the set of edges with sign $-1$ is $E^-=\{\{e,e'\}\mid e\leftrightarrow
    e'\text{ and }e\not\vartriangle e'\}$.
\end{enumerate}

The previous signed graph can be viewed as a combination of the Gallai graph and the standard signed line graph \cite{le}. Apparently, there are several definitions of a signed line graph, and a short discussion on them can be found in \cite{BSZ}. The adjacency matrix $A(\Lambda(G)) = (a_{ij}) $  is given by
$$
a_{ij}=
\begin{cases}
	\triangle(e)+2, & \mbox{if $i=j$}  \\
	1, & \mbox{if $ij \in E(\Lambda(G))$ with sign 1}  \\
	-1, & \mbox{if $ij \in E(\Lambda(G))$ with sign $-1$}\\
	0, & \text{otherwise}.
\end{cases}
$$
By deleting the loops from $\Lambda(G)$, we arrive at its reduction, denoted by $\Lambda_R(G)$. Here is an example in Fig. \ref{Ge9}.
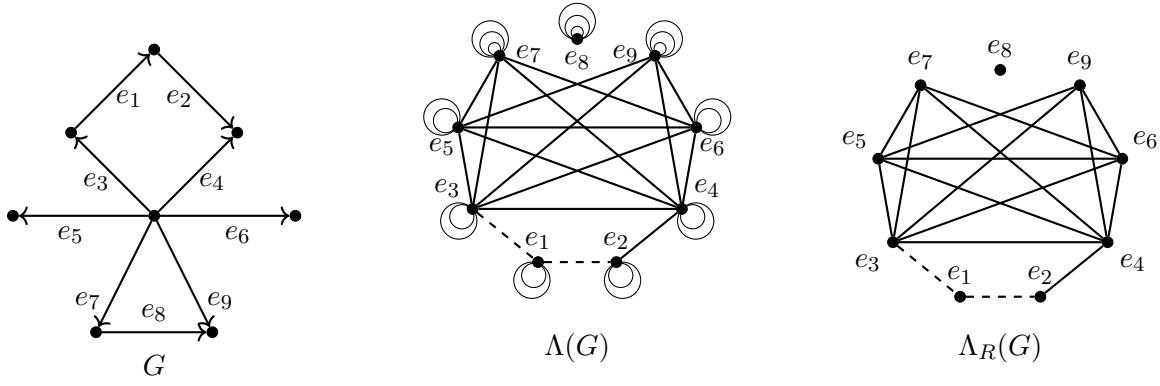
\begin{figure}[]
	\begin{minipage}{.3\textwidth}
		%\flushright
		\centering
		\begin{tikzpicture}[scale = 1.1]
			\node[draw,circle,inner sep=0.5mm,fill] (v1) at (0,1) {};
			\node[draw,circle,inner sep=0.5mm,fill] (v2) at (-1,0) {};
			\node[draw,circle,inner sep=0.5mm,fill] (v3) at (1,0) {};
			\node[draw,circle,inner sep=0.5mm,fill] (v4) at (0,-1) {};
			\node[draw,circle,inner sep=0.5mm,fill] (v5) at (-1.7,-1) {};
			\node[draw,circle,inner sep=0.5mm,fill] (v6) at (1.7,-1) {};
			\node[draw,circle,inner sep=0.5mm,fill] (v7) at (-0.7,-2.4) {};
			\node[draw,circle,inner sep=0.5mm,fill] (v8) at (0.7,-2.4) {};
			
			\draw [->,thick] (v2) edge node[pos=.4, right]{$e_{1}$}(v1);
			\draw [->,thick] (v1) edge node[pos=.6, left]{$e_{2}$}(v3);
			\draw [->,thick] (v4) edge node[pos=.4, left]{$e_{3}$} (v2);
			\draw [->,thick] (v4) edge node[pos=.4, right]{$e_{4}$} (v3);
			\draw [->,thick] (v4) edge node[pos=.6, below]{$e_{5}$} (v5);
			\draw [->,thick] (v4) edge node[pos=.6, below]{$e_{6}$} (v6);
			\draw [->,thick] (v4) edge node[pos=.75, left]{$e_{7}$} (v7);
			\draw [->,thick] (v4) edge node[pos=.75, right]{$e_{9}$} (v8);
			\draw [->,thick] (v7) edge node[pos=.5, above]{$e_{8}$}(v8);
			
			\node(98) at (0,1.6) {};
			\node(99) at (0,-2.8) {$G$};
		\end{tikzpicture}
	\end{minipage}%
	\hfill
	\begin{minipage}{.4\textwidth}
		%\flushleft
		\centering
		\begin{tikzpicture}[scale=0.8]
			
			\node[draw,circle,inner sep=0.5mm,fill] (v1) at (-0.6480,-1.8794) {};
			\draw (v1) node [above=1pt]{$e_1$};
			\node[draw,circle,inner sep=0.5mm,fill] (v2) at (0.6480,-1.8794) {};
			\draw (v2) node [above=1pt]{$e_2$};
			\node[draw,circle,inner sep=0.5mm,fill] (v3) at (-1.7320,-1) {};
			\draw (v3) node [above left=1pt]{$e_3$};
			\node[draw,circle,inner sep=0.5mm,fill] (v4) at (1.7320,-1) {};
			\draw (v4) node [above right=1pt]{$e_4$};
			\node[draw,circle,inner sep=0.5mm,fill] (v5) at (-1.9696,0.3472) {};
			\draw (v5) node [below left=0.1mm, xshift=1mm]{$e_5$};
			\node[draw,circle,inner sep=0.5mm,fill] (v6) at (1.9696,0.3472) {};
			\draw (v6) node [below right=0.1mm, xshift=-1mm]{$e_6$};
			\node[draw,circle,inner sep=0.5mm,fill] (v7) at (-1.2856,1.5320) {};
			\draw (v7) node [right=0.08cm, yshift=0.03cm]{$e_7$};
			\node[draw,circle,inner sep=0.5mm,fill] (v8) at (1.2856,1.5320) {};
			\draw (v8) node [left=0.08cm, yshift=0.03cm]{$e_9$};
			\node[draw,circle,inner sep=0.5mm,fill] (v9) at (0,1.81) {};
			\draw (v9) node [below=1pt]{$e_8$};
			
			\draw [thick,dashed] (v1) edge (v2);
			\draw [thick] (v2) edge (v4);
			\draw [thick] (v4) edge (v3);
			\draw [thick,dashed] (v3) edge (v1);
			\draw [thick] (v3) edge (v5);
			\draw [thick] (v5) edge (v6);
			\draw [thick] (v6) edge (v4);
			\draw [thick] (v7) edge (v5);
			\draw [thick] (v8) edge (v6);
			\draw [thick] (v4) edge (v7);
			\draw [thick] (v8) edge (v3);
			\draw [thick] (v3) edge (v7);
			\draw [thick] (v4) edge (v8);
			\draw [thick] (v5) edge (v4);
			\draw [thick] (v3) edge (v6);
			\draw [thick] (v7) edge (v6);
			\draw [thick] (v5) edge (v8);
			
			\draw  (1.37,1.65) ellipse (0.1 and 0.1);
			\draw  (1.41,1.73) ellipse (0.2 and 0.2);
			\draw  (1.44,1.81) ellipse (0.3 and 0.3);
			
			\draw  (-1.37,1.65) ellipse (0.1 and 0.1);
			\draw  (-1.41,1.73) ellipse (0.2 and 0.2);
			\draw  (-1.44,1.81) ellipse (0.3 and 0.3);
			
			\draw  (0,1.92) ellipse (0.1 and 0.1);
			\draw  (0,2) ellipse (0.2 and 0.2);
			\draw  (0,2.1) ellipse (0.3 and 0.3);
			
			\draw  (2.18,0.46) ellipse (0.2 and 0.2);
			\draw  (2.24,0.52) ellipse (0.3 and 0.3);
			
			\draw  (-2.18,0.46) ellipse (0.2 and 0.2);
			\draw  (-2.24,0.52) ellipse (0.3 and 0.3);
			
			\draw  (1.91,-1.13) ellipse (0.2 and 0.2);
			\draw  (1.96,-1.20) ellipse (0.3 and 0.3);
			
			\draw  (-1.91,-1.13) ellipse (0.2 and 0.2);
			\draw  (-1.96,-1.20) ellipse (0.3 and 0.3);
			
			\draw  (0.72,-2.10) ellipse (0.2 and 0.2);
			\draw  (0.75,-2.18) ellipse (0.3 and 0.3);
			
			\draw  (-0.72,-2.10) ellipse (0.2 and 0.2);
			\draw  (-0.75,-2.18) ellipse (0.3 and 0.3);

			\node(99) at (0,-3.3) {$\Lambda(G)$};
		\end{tikzpicture}
	\end{minipage}%
	\hfill
	\begin{minipage}{.3\textwidth}
		%\flushleft
		\centering
		\begin{tikzpicture}[scale=0.82]
			
			\node[draw,circle,inner sep=0.5mm,fill] (v1) at (-0.6480,-1.8794) {};
			\draw (v1) node [above=1pt]{$e_1$};
			\node[draw,circle,inner sep=0.5mm,fill] (v2) at (0.6480,-1.8794) {};
			\draw (v2) node [above=1pt]{$e_2$};
			\node[draw,circle,inner sep=0.5mm,fill] (v3) at (-1.7320,-1) {};
			\draw (v3) node [below left=1pt]{$e_3$};
			\node[draw,circle,inner sep=0.5mm,fill] (v4) at (1.7320,-1) {};
			\draw (v4) node [below right=1pt]{$e_4$};
			\node[draw,circle,inner sep=0.5mm,fill] (v5) at (-1.9696,0.3472) {};
			\draw (v5) node [above left=0.1pt]{$e_5$};
			\node[draw,circle,inner sep=0.5mm,fill] (v6) at (1.9696,0.3472) {};
			\draw (v6) node [above right=0.1pt]{$e_6$};
			\node[draw,circle,inner sep=0.5mm,fill] (v7) at (-1.2856,1.5320) {};
			\draw (v7) node [above=0.1pt, yshift=0.08cm]{$e_7$};
			\node[draw,circle,inner sep=0.5mm,fill] (v8) at (1.2856,1.5320) {};
			\draw (v8) node [above=0.1pt, yshift=0.08cm]{$e_9$};
			\node[draw,circle,inner sep=0.5mm,fill] (v9) at (0,1.78) {};
			\draw (v9) node [above=1pt]{$e_8$};
			
			\draw [thick,dashed] (v1) edge (v2);
			\draw [thick] (v2) edge (v4);
			\draw [thick] (v4) edge (v3);
			\draw [thick,dashed] (v3) edge (v1);
			\draw [thick] (v3) edge (v5);
			\draw [thick] (v5) edge (v6);
			\draw [thick] (v6) edge (v4);
			\draw [thick] (v7) edge (v5);
			\draw [thick] (v8) edge (v6);
			\draw [thick] (v4) edge (v7);
			\draw [thick] (v8) edge (v3);
			\draw [thick] (v3) edge (v7);
			\draw [thick] (v4) edge (v8);
			\draw [thick] (v5) edge (v4);
			\draw [thick] (v3) edge (v6);
			\draw [thick] (v7) edge (v6);
			\draw [thick] (v5) edge (v8);
			
			\node(98) at (0,2.7) {};
			\node(99) at (0,-2.7) {$\Lambda_R(G)$};
		\end{tikzpicture}
	\end{minipage}
	\caption{The graph $G$,  $\Lambda(G)$ and $\Lambda_R(G)$.}
	\label{Ge9}
\end{figure}

In the light of Theorem~\ref{H-def}, we have
\begin{equation*}
\label{eq:rel}\mathcal{H}(G)=A(\Lambda(G))=A(\Lambda_R(G))+\mathcal{D}(G),
\end{equation*}
where $\mathcal{D}(G)=\diag(\triangle(e)+2\mid e\in E(G))$. Then, $\phi_{\mathcal{H}}(G,\lambda)$ is just the characteristic polynomial of $A(\Lambda(G))$.

In  $\Lambda(G)$, the {\it weight} $w(v)$ of a vertex $v$ is the number of loops attached at $v$, the weight $w^+(e)$ (or $w^-(e)$) of an edge $e$ is the positive (or negative) sign of $e$, and the weight $w(C)$ of a cycle $C$ is the product of  weights of its edges. We also write $c^o_+$, $c^o_-$, $c^e_+$ and $c^e_-$ for the number of positive odd cycles, negative odd cycles, positive even  cycles and negative even cycles in $\Lambda(G)$, respectively.

We need more notation. A subgraph is called an {\it elementary graph} if it is an isolated vertex $K_1$, or an isolated edge $K_2$, or a cycle $C_n$ ($n \geq 3$). The subgraph $B$ of $\Lambda(G)$ containing $k$ vertices is a {\it basic subgraph} if all  its components are elementary graphs. Let $n(B)$, $m(B)$ and $c(B)$ be the number of isolated vertices, the number of isolated edges and the number of cycles in $B$, respectively. Let $p(B)$ be the number of non-trivial components, i.e., those distinct from $K_1$. Therefore, $$p(B)=m(B)+c(B) \quad \text{and} \quad c(B)=c^o_++c^o_-+c^e_++c^e_-.$$
In a basic subgraph $B$, let $\mathfrak{B}_v$, $\mathfrak{B}_e$ and $\mathfrak{B}_c$ be the sets of isolated vertices, isolated edges and cycles, respectively. The {\it vertex-weight}, {\it edge-weight} and {\it cycle-weight} of $B$ are  $$\varpi_v(B) = \prod\limits_{v \in \mathfrak{B}_v}w(v),\;\, \varpi_e(B) = (-1)^{m(B)} \quad \mbox{and} \quad \varpi_c(B) = \prod\limits_{C \in \mathfrak{B}_c}(-1)^{c^o_-(B)+c^e_+(B)}2^{c(B)}.$$ Set $\varpi_v(B) = \varpi_c(B) = 1$  if $\mathfrak{B}_v = \mathfrak{B}_c = \emptyset$. Then the resulting {\it total weight} of $B$ is defined as
$$\varpi(B) = \varpi_v(B)\varpi_e(B)\varpi_c(B) = (-1)^{m(B)+c^o_-(B)+c^e_+(B)}2^{c(B)}\varpi_v(B).$$

\begin{thm}\label{thm-2}
With the notation above, it holds	\[c_k=(-1)^k\sum_{B\in\mathcal{B}_k}\varpi(B)=(-1)^k\sum_{B\in\mathcal{B}_k}(-1)^{m(B)+c^o_-(B)+c^e_+(B)}2^{c(B)}\varpi_v(B),\]
	where $\mathcal{B}_k$ is the set of all basic subgraphs of $\Lambda(G)$ containing $k$ vertices.
\end{thm}

\begin{proof}
	Since $\phi_{\mathcal{H}}(G,\lambda)=\det(\lambda I_m -A(\Lambda(G)))$, we have that $(-1)^kc_k$ is just the sum of all $k\times k$ principal minors of $A(\Lambda(G))$. Take a $k\times k$ principal submatrix $M$ from $A(\Lambda(G))$, and assume that the rows and columns of $M$ are indexed by $\{v_1,v_2,\ldots,v_k\}$. Hence, we have
	\[\det (M)=\sum_{\tau}{\sgn}(\tau)M_{1,\tau(1)}M_{2,\tau(2)}\cdots M_{k,\tau(k)},\]
	where $\tau$ is a permutation of $\{1,2,\ldots,k\}$ and ${\sgn}(\tau)$ is the sign of $\tau$. Clearly, $\tau$ is a product of some disjoint cycles. Thus, for each transposition $(i,j)$ in $\tau$ the contribution of $M_{i,j}M_{ji}$ in the term $M_{1,\tau(1)}M_{2,\tau(2)}\cdots M_{k,\tau(k)}$ it is $1$ if $v_i\sim v_j$, and $0$ otherwise. Each cycle $(i_1,i_2,\ldots,i_l)$ contributes $M_{i_1,i_2}M_{i_2,i_3}\cdots M_{i_{l-1},i_l}M_{i_l,i_1}$ in the term $M_{1,\tau(1)}M_{2,\tau(2)}\cdots M_{k,\tau(k)}$ is $w(C)$ if $v_{i_1},v_{i_2},\ldots,v_{i_l}$ forms a cycle, and $0$ otherwise. Each fixed point $i$ contributes $M_{i,i}$, that is  $w(v_i)$. Thus, a term $M_{1,\tau(1)}M_{2,\tau(2)}\cdots M_{k,\tau(k)}$ is non-zero only if there is a basic subgraph, say $B$, corresponding to $\tau$. In this case, this term is exactly
	\[(-1)^{m(B)}\big(\prod_{C}w(C)\big)\big(\prod_{v_x}w(v_x)\big)=(-1)^{c^o_-+c^e_-}\varpi_v(B),\]
	where $C$ ranges over all cycles in $B$ and $v_x$ ranges over all isolated vertices in $B$. Note that a transposition contributes $-1$ in ${\sgn}(\tau)$, an odd cycle contributes $1$ in ${\sgn}(\tau)$ and an even cycle contributes $-1$ in ${\sgn}(\tau)$. Thereby, ${\sgn}(\tau)=(-1)^{m(B)}(-1)^{c^e_++c^e_-}=(-1)^{m(B)+c^e_++c^e_-}$, and thus
	\[{\sgn}(\tau)M_{1,\tau(1)}M_{2,\tau(2)}\cdots M_{k,\tau(k)}=(-1)^{m(B)+c^o_-(B)+c^e_+(B)}\varpi_v(B).\]
	On the other hand, each cycle in $B$ gives rise to two  permutations obtained by reversing this cycle, which do not change the value of the term ${\sgn}(\tau)M_{1,\tau(1)}M_{2,\tau(2)}\cdots M_{k,\tau(k)}$. Thus, in total, each $B$ provides $2^{c(B)}$  permutations which share the same value. Hence, we have
	\[\det(M)=\sum_{B}(-1)^{m(B)+c^o_-(B)+c^e_+(B)}2^{c(B)}\varpi_v(B),\]
	where $B$ ranges over all basic subgraphs with vertex set $\{v_1,v_2,\ldots,v_k\}$. Hence, by considering all $k\times k$ principal minors of $A(\Lambda(G))$, we obtain
	\[c_k=(-1)^k\sum_{B\in\mathcal{B}_k}(-1)^{m(B)+c^o_-(B)+c^e_+(B)}2^{c(B)}\varpi_v(B).\]
	This finishes the proof.
\end{proof}

\begin{re}\label{re-2}
	Observe that there are at least two loops attached at each vertex of $\Lambda(G)$, which significantly enlarges the number of
	basic subgraphs on $k$ vertices, because every set of $k$ isolated vertices forms a basic
	subgraph. If we delete two loops at every vertex, then the number of basic subgraphs falls down, and so let $\Lambda(G)'$ be obtained from $\Lambda(G)$ in this way. It is clear that $\lambda$ is an eigenvalue of $A(\Lambda(G))$ if and only if
	$\lambda-2$ is an eigenvalue of $A(\Lambda(G)')$. Consequently, it would be convenient to compute the
	eigenvalues $A(\Lambda(G))$ by computing the eigenvalues of $A(\Lambda(G)')$.
\end{re}

We proceed to compute the first three coefficients of the $\mathcal{H}$-polynomial. For a vertex $u$ of a graph $G$, the {\it triangle degree} of $u$, denoted by $\triangle_G(u)$, is the number of triangles containing  $u$.  The {\it neighbourhood} of an edge $e \in E(G)$, denoted by $N_G(e)$, is a set of edges adjacent to $e$.

\begin{cor}
	The first three coefficients of the $\mathcal{H}$-polynomial of a graph $G$:
	\begin{itemize}
		\item[{(i)}]
		$c_1=-\sum\limits_{e\in E(G)}\triangle_G(e)-2m(G)$;\\
		\item[{(ii)}]
		$c_2=\sum\limits_{e_i,e_j \in E(G)}(\triangle_G(e_i)+2)(\triangle_G(e_j)+2)-\sum\limits_{v \in V(G)}{d_G(v) \choose 2} + 3t_G(\triangle)$;\\
		\item[{(iii)}] $c_3 = c_{31}+c_{32}+c_{33}$, where
		$$
		\begin{aligned}
			c_{31}&=-\sum\limits_{e_u,e_v,e_z \in E(G)}(\triangle_G(e_u)+2)(\triangle_G(e_v)+2)(\triangle_G(e_z)+2);\\
			c_{32}&=\sum\limits_{e\in E(G)}\left(\sum\limits_{u \in V(G)}\binom{d_G(u)}{2}-3t_G(\triangle)-N_G(e)+2\triangle_G(e)\right)\left(\triangle_G(e)+2\right)\\
			c_{33}&=-2 \sum_{u \in V(G)}\left(\binom{d_G(u)}{3}-\triangle_G(u)(d_{G}(u)-2)\right).
		\end{aligned}
		$$
	\end{itemize}
\end{cor}

\begin{proof}
Let $\Line(G)$ and $\Lambda_R(G)$ denote the line graph of $G$ and the reduction of $\Lambda(G)$ defined above. Obviously, a basic subgraph of order $1$ is a vertex of $\Lambda(G)$. In terms of Theorem \ref{thm-2}, $$c_1=-\sum\limits_{K_1 \in \mathcal{B}_1}\varpi(K_1) = -\sum\limits_{e\in E(G)}\triangle_G(e)-2m(G).$$
	
	For $c_2$,  basic subgraphs of $\Lambda(G)$ in $\mathcal{B}_2$ consist of two kinds of  elementary subgraphs. One of them is  $B\cong 2K_1$, and its  contribution to $c_2$ is
	$$\sum\limits_{B \in \mathcal{B}_2 }\varpi(2K_1) = \sum\limits_{u,v\in \Lambda(G)}\varpi(u)\varpi(v)=\sum\limits_{e_u,e_v \in E(G)}(\triangle_G(e_u)+2)(\triangle_G(e_v)+2).$$ The other one is an isolated edge $B\cong K_2$ whose contribution to $c_2$ is  $$\sum\limits_{B \in \mathcal{B}_2 }\varpi(K_2) = -m(\Lambda_R(G)) = -m(\Line(G))+3t_G(\triangle) = -\sum\limits_{v \in V(G)}{d_G(v) \choose 2} + 3t_G(\triangle).$$  Hence,
	$$c_2 = \sum\limits_{e_u,e_v \in E(G)}(\triangle_G(e_u)+2)(\triangle_G(e_v)+2)-\sum\limits_{v \in V(G)}{d_G(v) \choose 2} + 3t_G(\triangle).$$
	
	We next compute $c_3$. We orient the edges in such a way that for the any triangle $K_3 \subseteq \Lambda(G)$, we have $w(K_3) =1$. The basic subgraphs of $\Lambda(G)$ in $\mathcal{B}_3$ are composed of three kinds of elementary subgraphs. The first is $3K_1$ whose contribution to $c_3$ is
	\begin{align*}
		\sum\limits_{3K_1 \in \mathcal{B}_3}\varpi(3K_1) &=\sum\limits_{u,v,z\in \Lambda(G)}\varpi(u)\varpi(v)\varpi(z)\\
		& = \sum\limits_{e_u,e_v,e_z \in E(G)}(\triangle_G(e_u)+2)(\triangle_G(e_v)+2)(\triangle_G(e_z)+2).
	\end{align*}
	The second one is $K_1 \cup K_2$ whose contribution to $c_3$ is
	$$
	\begin{aligned}
		\sum\limits_{K_1\cup K_2\in\mathcal{B}_2} \varpi(K_{1} \cup K_{2})
		&=-\sum\limits_{\substack{e \in \Lambda(G)\\ u \notin e}}\varpi(u)\\
		&=-\sum\limits_{e\in E(G)}\left(m(\Line(G))-3t_G(\triangle)-(N_G(e)-2\triangle_G(e))\right)(\triangle_G(e)+2)\\
		&=-\sum\limits_{e\in E(G)}\left(\sum\limits_{u \in V(G)}\binom{d_G(u)}{2}-3t_G(\triangle)-N_G(e)+2\triangle_G(e)\right)\left(\triangle_G(e)+2\right).
	\end{aligned}
	$$
	The third one is $K_3$ whose contribution to $c_3$ is
	$$\begin{aligned}
		\sum_{K_{3} \in \mathcal{B}_{3}} \varpi(K_{3})
		&=\sum_{K_{3} \in \mathcal{B}_{3}} 2\\
		&=2 \sum_{u \in V(G)}\left(\binom{d_G(u)}{3}-\binom{\triangle_G(u)}{1}\binom{d_G(u)-2}{1}\right)\\
		&=2 \sum_{u \in V(G)}\left(\binom{d_G(u)}{3}-\triangle_G(u)(d_{G}(u)-2)\right).
	\end{aligned}
	$$
	Summing up the previous values, we arrive at the desired result.
\end{proof}

Here is an example.

\begin{examplex}
	The Helmholtzian  of the graph illustrated in Fig.~\ref{Ge9},  is
	$$
	\mathcal{H}(G)=
	\begin{pmatrix}
		2 & -1 & -1 & 0 & 0 & 0 & 0 & 0 & 0 \\
		-1 & 2 & 0 & 1 & 0 & 0 & 0 & 0 & 0 \\
		-1 & 0 & 2 & 1 & 1 & 1 & 1 & 0 & 1 \\
		0 & 1 & 1 & 2 & 1 & 1 & 1 & 0 & 1 \\
		0 & 0 & 1 & 1 & 2 & 1 & 1 & 0 & 1 \\
		0 & 0 & 1 & 1 & 1 & 2 & 1 & 0 & 1 \\
		0 & 0 & 1 & 1 & 1 & 1 & 3 & 0 & 0 \\
		0 & 0 & 0 & 0 & 0 & 0 & 0 & 3 & 0 \\
		0 & 0 & 1 & 1 & 1 & 1 & 0 & 0 & 3
	\end{pmatrix}.
	$$
	The corresponding characteristic polynomial is computed directly:
	$$\phi_{\mathcal{H}}(G,\lambda) = \lambda^9-21\lambda^8+178\lambda^7-802\lambda^6+2105\lambda^5-3293\lambda^4+
	2996\lambda^3-1452\lambda^2+288\lambda.$$
	On the other hand, we compute the first three coefficients by employing the previous corollary. Accordingly,
	$$
	\begin{aligned}
		c_1&=-\sum\limits_{e \in E(G)}\triangle_{G}(e)-2m(G)=-3-18=-21;\\
		c_{2}&=\sum\limits_{e_{i},e_{j}\in E(G)} (\triangle_{G}(e_{i})+2)(\triangle_{G}(e_{j})+2)-\sum_{v \in V(G)}\binom{d_G(v)}{2}+3 t_{G}(\triangle)\\
		&=195-20+3\\
		&=178.
	\end{aligned}
	$$
	For $c_3$, since
	$$
	\begin{aligned}
		c_{31}&=-\sum\limits_{e_u,e_v,e_z \in E(G)}(\triangle_G(e_u)+2)(\triangle_G(e_v)+2)(\triangle_G(e_z)+2)=-1051;\\
		c_{32}&=\sum\limits_{e\in E(G)}\left(\sum\limits_{u \in V(G)}\binom{d_G(u)}{2}-3t_G(\triangle)-N_G(e)+2\triangle_G(e)\right)\left(\triangle_G(e)+2\right)=281;\\
		c_{33}&=-2 \sum_{u \in V(G)}\left(\binom{d_G(u)}{3}-\triangle_G(u)(d_{G}(u)-2)\right)=-32,
	\end{aligned}
	$$
	we have $c_3=-1051+281-32=-802$.
\end{examplex}

In Theorem \ref{thm-2}, we provide a combinatorial expression for the coefficients of the $\mathcal{H}$-polynomial of a graph $G$, but this expression depends on the associated signed graph $\Lambda(G)$. Therefore, we look forward to a more direct expression.

\begin{problem}
Provide a combinatorial interpretation of the coefficients of the $\mathcal{H}$-polynomial of a graph $G$ that depends only on the subgraph structures of $G$.
\end{problem}

\subsection{Graph Products}\label{H49}

Products of graphs (or networks) play an important role in the design and analysis of networks. Their spectra represent a powerful tool revealing the global structural patterns underlying a large-scale complex network of interacting entities, see \cite{ham-imr-kla,mie}. In this section we first consider the join $G_1 \vee G_1$  of two graphs $G_1$ and $G_2$.

\begin{prop}\label{lem-gp-1}
	Let $G_1$ and $G_2$ be  graphs or order $n_1$ and $n_2$, respectively.  Then
	\[\mathcal{H}(G_1 \vee G_2) =
	\left(\begin{array}{ccc}\mathcal{H}(G_1)+n_2I&O&O\\ O&\mathcal{H}(G_2)+n_1I&O\\O&O&I\otimes A(\overline{G}_2)+A(\overline{G}_1)\otimes I+X\end{array}\right),\]
	where $A(\overline{G}_i)$ is the adjacency matrix of the complementary  graph $\overline{G}_i$, and $X$ is the diagonal matrix indexed by $V(G_1)\times V(G_2)$ such that the diagonal entry corresponding to $(u,v)$ is $d_{G_1}(u)+d_{G_2}(v)+2$.
\end{prop}

\begin{proof}
	Set $V(G_1)=\{u_1,u_2,\ldots,u_{n_1}\}$ and $V(G_2)=\{v_1,v_2,\ldots,v_{n_2}\}$. Let $E_1=E(G_1)$, $E_2=E(G_2)$, and $E_{3}=\{u_iv_j\mid 1\le i\le n_1, 1\le j\le n_2\}$. Without loss of generality, we assume that $u_i\rightarrow v_j$ holds for every $u_i\in V(G_1)$ and $v_j\in V(G_2)$. Denote by $X_{ij}$ the block of $\mathcal{H}(G_1 \vee G_2)$ whose rows are indexed by $E_i$ and columns indexed by $E_j$ for $1\le i,j\le 3$.
	
	 For an edge $e\in E_1$, we have $\triangle_{G_1\vee G_2}(e)=\triangle_{G_1}(e)+n_2$. Therefore, we have $X_{11}=\mathcal{H}_1(G)+n_2I$, and similarly $X_{22}=\mathcal{H}_2(G)+n_1I$.
	
	For the edges $e_1\in E_1$ and $e_2\in E_2$, $X_{12}=X_{21}^{\intercal}=O$ holds since $V(e_1)\cap V(e_2)=\emptyset$.
	
	For the edges $e_1=u_iu_j\in E_1$ and $e_3=u_xv_y\in E_{3}$, if $x\not\in\{i,j\}$, then the $(e,e')$-entry of $\mathcal{H}(G_1 \vee G_2)$ is $0$; if $x=i$, then $\{u_i,u_j,v_y\}$ forms a triangle, and thus the $(e,e')$-entry of $\mathcal{H}(G_1 \vee G_2)$ is still $0$. Therefore, $X_{13}=X_{31}^{\intercal}=O$ and similarly $X_{23}=X_{32}^{\intercal}=O$.
	
	For the edges $e=u_iv_j,e'=u_xv_y\in E_3$, if $i=x$ and $j=y$, then it is easy to see that $\triangle_{G_1\nabla G_2}(e)=d_{G_1}(u_i)+d_{G_2}(v_j)$ and the $(e,e)$-entry of $\mathcal{H}(G_1 \vee G_2)$ is $d_{G_1}(u_i)+d_{G_2}(v_j)+2$. If $i\ne x$ and $j\ne y$, then the $(e,e')$-entry of $\mathcal{H}(G_1 \vee G_2)$ is $0$. If $i=x$ and $j\ne y$, then the $(e,e')$-entry is $1$ if and only if $\{u_i,v_j,v_y\}$ does not form a triangle, i.e., $v_j\not\sim v_y$ in $G_2$. If $i\ne x$ and $j=y$, then the $(e,e')$-entry is $1$ if and only if $\{u_i,u_x,v_j\}$ does not form a triangle, i.e., $u_i\not\sim u_x$. Thereby, if $e\ne e'$, then the $(e,e')$-entry of $\mathcal{H}(G_1 \vee G_2)$ is $1$ if and only if $u_i=u_x$ and $v_j\sim v_y$ in $\overline{G}_2$, or $v_j=v_y$ and $u_i\sim u_x$ in $\overline{G}_1$. This leads to $$X_{33}-X=I\otimes A(G_2)+A(G_1)\otimes I,$$ where $X$ is the diagonal matrix indexed by $V(G_1)\times V(G_2)$ with  $(u,v)$-diagonal entry equal to $d_{G_1}(u)+d_{G_2}(v)+2$.
	
	The proof is completed.
\end{proof}

We can say more in case of regular graphs.

\begin{thm}\label{the:join}
	Let $G_i$ be an $r_i$-regular graph of order $n_i$ and size $m_i$ ($1\le i\le 2$). If the $\mathcal{H}$-eigenvalues of $G_i$ are $\lambda_{i1} \ge\lambda_{i2}\ge \cdots \ge \lambda_{im_i}$ and the eigenvalues of the adjacency matrix of $G_i$ are $r_i=\mu_{i1} \ge \mu_{i2} \ge \cdots \ge \mu_{in_i}$, then
	\[{\Sp}_{\mathcal{H}}(G_1 \vee G_2) = \left\{\lambda_{1t_1}+n_2, \lambda_{2t_2}+n_1, n_1+n_2,n_1+r_2-\mu_{2k},n_2+r_1-\mu_{1j},r_1+r_2-\mu_{1j}-\mu_{2k}\right\},\]
	where $1\le t_1\le m_1$,  $1\le t_2\le m_2$, $2\le j\le n_1$, and $2\le k\le n_2$.
\end{thm}

\begin{proof}
	By Proposition \ref{lem-gp-1}, we have
	\begin{equation}\label{eq:joinn}\mathcal{H}(G_1\vee G_2)=\left(\begin{array}{ccc}
		\mathcal{H}(G_1)+n_2I&O&O\\
		O&\mathcal{H}(G_2)+n_1I&O\\
		O&O&I\otimes A(\overline{G}_2)+A(\overline{G}_1)\otimes I+(r_1+r_2+2)I
	\end{array}\right).\end{equation}
	For $1\le i\le 2$, observing that $A(G_i)\mathbf{1}=r_i\mathbf{1}$,  we deduce $A(\overline{G}_i)\mathbf{1}=(n_i-1-r_i)\mathbf{1}$. For $1\le i\le 2$ and $2\le j\le n_i$, assume that ${\bf x}_{ij}$ are orthogonal eigenvectors of $A(G_i)$ corresponding to $\mu_{ij}$, that is $A(G_i){\bf x}_{ij}=\mu_{ij}{\bf x}_{ij}$. Clearly,  $A(\overline{G}_i){\bf x}_{ij} = (-1-\mu{ij}){\bf x}_{ij}$. Hence,
	\[\left\{
	\begin{array}{l}
		(I\otimes A(\overline{G}_2)+A(\overline{G}_1)\otimes I)(\mathbf{1}\otimes \mathbf{1})=(n_1+n_2-2-r_1-r_2)(\mathbf{1}\otimes \mathbf{1}),\\[2mm]
		(I\otimes A(\overline{G}_2)+A(\overline{G}_1)\otimes I)(\mathbf{1}\otimes {\bf x}_{2k})=(n_1-2-r_1-\mu_{2j})(\mathbf{1}\otimes {\bf x}_{2k}),\\[2mm]
		(I\otimes A(\overline{G}_2)+A(\overline{G}_1)\otimes I)({\bf x}_{1j} \otimes \mathbf{1})=(n_2-2-r_2-\mu_{1j})({\bf x}_{1j}\otimes \mathbf{1}),\\[2mm]
		(I\otimes A(\overline{G}_2)+A(\overline{G}_1)\otimes I)({\bf x}_{1j}\otimes {\bf x}_{2k})=(-2-\mu_{1j}-\mu_{2k})({\bf x}_{1j}\otimes {\bf x}_{2k}),\\
	\end{array}
	\right.\]
	which shows that the $\mathcal{H}$-eigenvalues of $I\otimes A(\overline{G}_2)+A(\overline{G}_1)\otimes I+(r_1+r_2+2)I$ are
	\[\left\{n_1+n_2,n_1+r_2-\mu_{2k},n_2+r_1-\mu_{1j},r_1+r_2-\mu_{1j}-\mu_{2k}\right\}.\]
	Therefore, the $\mathcal{H}$-spectrum of $G_1 \vee G_2$ is
	as in the statement formulation.
\end{proof}

The previous proof was based on an explicit computation of the corresponding eigenvectors. In this context, we remark that the same result can be obtained by noting that the summands in the bottom-right matrix of \eqref{eq:joinn} mutually commute, and so they can be simultaneously diagonalized.

We finally consider a more general graph operation usually called the NEPS (an acronym of the {\it non-complete extended $p$-sum}) of graphs (see, for instance, \cite{cve-row-sim-book, cve-doob-sachs-book}). Let
$\mathscr{B} \subseteq  \Z_{2}^{n} \setminus (0,0,\ldots ,0)$. The NEPS of graphs $G_1,G_2,\ldots ,G_n$ with basis $\mathscr{B} $ is
the graph with vertex set $V(G_1)\times V(G_2)\times \cdots \times V(G_n)$, in which two vertices, say $(x_1, x_2, \ldots ,x_n)$ and $(y_1,y_2, \ldots ,y_n)$
are adjacent if and only if there exists an $n$-tuple $\beta =(\beta_1,\beta_2, \ldots ,\beta_n)\in  \mathscr{B}$ such that $x_i =y_i$
whenever $\beta_i =0$, and $x_i$ is adjacent to $y_i$ (in $G_i$) whenever $\beta_i =1$.

The NEPS construction generates many binary graph products (see \cite[Section 2.5]{cve-doob-sachs-book}). In particular, for $n=2$ we have
the following familiar operations:
\begin{itemize}
	\item[(i)]
	the {\it Cartesian product} $G_1 \Box G_2$, when $\mathscr{B} =\{(0,1),(1,0)\}$;
	\item[(ii)]
	the {\it direct product} $G_1 \times G_2$, when $\mathscr{B} =\{(1,1)\}$;
	\item[(iii)]
	the {\it strong product} $G_1 \boxtimes G_2$, when $\mathscr{B} =\{(0,1),(1,0),(1,1)\}$.
\end{itemize}

So, we propose the following problem for further study.

\begin{problem}
Determine	under which conditions the $\mathcal{H}$-eigenvalues of a NEPS of $G_1,G_2,\ldots,G_n$ with basis $\mathcal{B}$ can be obtained. In particular, compute the $\mathcal{H}$-eigenvalues of the products (i)--(iii) in therms of the $\mathcal{H}$-eigenvalues of $G_1$ and $G_2$.
\end{problem}

\subsection{$\mathcal{H}$-integral Graphs}\label{H41}

A graph $G$ is called {\it $\mathcal{H}$-integral} if all the $\mathcal{H}$-eigenvalues are integers. Of course, there is an analogous definition for any other matrix associated with graphs, and there is an intensive study of graph with integral spectrum of the adjacency matrix or the Laplacian matrix. To the best of our knowledge, this topic was opened by Harary and Schwenk in 1974~\cite{har-sch}.

Certain examples of $\mathcal{H}$-integral graphs are met before this subsection, see the complete graphs and generalized windmill graphs in Subsection  \ref{H48}. An other example is obtained by taking two regular graphs with integral spectrum and integral $\mathcal{H}$-spectrum, and then applying the join operation. The resulting graph is $\mathcal{H}$-integral by Theorem~\ref{the:join}. In this regard, we note that many results concerning integral regular graphs can be found in \cite[Chapter~4]{rgasa}.
In what follows we determine some infinite families $\mathcal{H}$-integral graphs.

A {\it complete $k$-partite graph} is a graph whose vertices are partitioned into $k$ sets in such a way that two vertices are adjacent if and only if they do not belong to the same set. This graph is denoted by $K_{n_1,n_2,\ldots,n_k}$, where $n_i$ is the size of the $i$th set.  Let's start with a complete $2$-partite graph $K_{n_1,n_2}$. For later purposes, we note that its Helmholtzian matrix is
	\begin{equation}\label{HKn1n2}
		\mathcal{H}(K_{n_1,n_2})
		=\begin{pmatrix}
			J+I & I & \cdots & I \\
			I & J+I & \cdots & I \\
			\vdots & \vdots & \ddots & \vdots \\
			I & I & \cdots & J+I
		\end{pmatrix}_{n_1n_2\times n_1n_2}.
\end{equation}
On one hand, its Laplacian spectrum is (see \cite{bro-hae-book})
$$
{\Sp}_L(K_{n_1,n_2}) =
\left\{\!\!\begin{array}{cccc}
	n_1+n_2 & n_1 & n_2 & 0 \\
	1 & n_2-1 & n_1-1 & 1
\end{array}
\!\!\right\}.
$$
On the other hand,  the Laplacian matrix  $L(G) = \mathcal{B}^{\intercal}\mathcal{B}$. When $G$ is triangle-free, from \eqref{H=B+C} we obtain $\mathcal{H}(G) = \mathcal{B}\mathcal{B}^{\intercal}$.  Consequently, the $\mathcal{H}$-eigenvalues of $G$ are formed from  its non-zero Laplacian eigenvalues along with $m-n+1$ zero eigenvalues. Thereby,
\begin{equation}\label{HKn1n2-sp}
	{\Sp}_{\mathcal{H}}(K_{n_1,n_2}) =
	\left\{\!\!\begin{array}{cccc}
		n_1+n_2 & n_1 & n_2 & 0 \\
		1 & n_2-1 & n_1-1 & n_1n_2-n_1-n_2+1
	\end{array}
	\!\!\right\}.
\end{equation}

Now, we turn on to the general case. In our approach, the union symbol between multisets refers to the multiset in which the multiplicity of an element is the sum of its multiplicities in the corresponding multisets.

\begin{prop}\label{H-Kn1nk}
	Let $G = K_{n_1,n_2,\ldots,n_k}$ be the complete $k$-partite graph of order $n = \sum_{i=1}^kn_i$. Then $G$ $\mathcal{H}$-integral. Moreover,
	{\footnotesize
		$$
		\begin{aligned}
			&{\Sp}_{\mathcal{H}}(G) =
			\left\{\!\!\begin{array}{cccc}
				n                 & n-n_1        & \cdots & n-n_k       \\
				\frac{k(k-1)}{2} & (k-1)(n_1-1) & \cdots &(k-1)(n_k-1)
			\end{array}
			\!\!\right\} \bigcup\\
			&{\footnotesize
				\left\{\!\!\begin{array}{cccccc}
					n-n_1-n_2        & \cdots & n-n_1-n_k        & n-n_2-n_3       & \cdots & n-n_{k-1}-n_k  \\
					n_1n_2-n_1-n_2+1 & \cdots & n_1n_k-n_1-n_k+1 & n_2n_3-n_2-n_3+1 & \cdots & n_{k-1}n_k-n_{k-1}-n_k
				\end{array}
				\!\!\right\}.}
		\end{aligned}
		$$}
\end{prop}

\begin{proof}
	We set $\mathcal{H}(K_{n_1,n_2,\dots,n_k})=\big(h_{ee^{\prime}}\big)$ and denote the vertex set partition by $V=V_1\cup V_2\cup\cdots\cup V_k$ with $|V_i| = n_i$ ($1 \leq i \leq k$). For any two vertices $u \in V_i$ and $v \in V_j$ ($1 \leq i < j \leq k$), we give an orientation such that $u \rightarrow v$.
	If the three endpoints of two adjacent edges $e\sim e^{\prime}$ are located in only two vertex subsets, then $e$ and  $e^{\prime}$ are not in a same triangle which leads to $h_{ee^{\prime}}=1$ by Theorem \ref{H-def}; if  the endpoints are in three vertex subsets, then $e$ and  $e^{\prime}$ belong to a common triangle resulting in $h_{ee^{\prime}}=0$; if $\mathcal{V}(e)\bigcap \mathcal{V}(e^{\prime})=\emptyset$, then $h_{ee^{\prime}}=0$. An edge located between $u \in V_i$ and $v \in V_j$   is contained in $n-n_i-n_j$ triangles  indicating  $h_{ee}=n-n_i-n_j+2$. Therefore, we arrive at a block diagonal matrix
	\begin{equation*}
		\mathcal{H}(K_{n_1,n_2,\ldots,n_k})=
		{
			\begin{pmatrix}
				H_{12} &  &  &  &  &  \\
				& \ddots &  &  &  &  \\
				&  & H_{1k}&  &  &  \\
				&  &  & H_{23} &  &  \\
				&  &  &  & \ddots &  \\
				&  &  &  &  & H_{k-1,k}
			\end{pmatrix},
		}
	\end{equation*}
	where
	$${
		H_{ij}=\begin{pmatrix}
			N & I & \cdots & I \\
			I & N & \cdots & I \\
			\vdots & \vdots & \ddots & \vdots \\
			I & I & \cdots & N
	\end{pmatrix}}_{n_in_j}\!\!\!\!\!\!, \quad N=J+(n-n_i-n_j+1)I, \,\; \text{with}\; N\in\mathbb{R}^{n_j\times n_j}\;\; (1 \leq i < j \leq k).
	$$
	On the basis of \eqref{HKn1n2} and \eqref{HKn1n2-sp}, we obtain
	\begin{equation*}\label{HKn1n2-sp1}
		{\Sp}(H_{ij}) =
		\left\{\!\!\begin{array}{cccc}
			n & n-n_i & n-n_j & n-n_i-n_j \\
			1 & n_i-1 & n_j-1 & n_in_j-n_j-n_j+1
		\end{array}
		\!\!\right\}, \; \mbox{with}\; 1 \leq i < j \leq k.
	\end{equation*}
	This ends the proof.
\end{proof}

We proceed with threshold graphs. A \textit{threshold graph} is a $\{2K_2, P_4, C_4\}$-free graph, i.e., it  does not contain  any of listed graphs as an induced subgraph. According to \cite{mah-pel}, every threshold graph can be constructed from a single vertex
by repeated  the following operations: (1) addition of a single isolated
vertex to the graph; (2) addition of a single dominating vertex to the graph, i.e., a single
vertex that is joined by an edge to all other vertices. If the addition of an isolated
vertex is denoted by 0 and the addition of a dominating vertex by 1, then we say that
$G=G(\mathbf{b})$ is generated by a binary code $\mathbf{b}=(b_1, b_2, \ldots, b_n)$.

We give an iterative  formula for the $\mathcal{H}$-spectrum of a threshold graph.

\begin{prop}\label{prop:thresh}Let $G=G(b_1, b_2, \ldots, b_n)$ be a threshold graph. If $\lambda_1, \lambda_2, \ldots, \lambda_m$ and $\mu_1\geq\mu_2\geq \cdots \geq \mu_n=0$ are the $\mathcal{H}$-eigenvalues and the Laplacian eigenvalues of $G(b_1, b_2, \ldots, b_{n-1})$ respectively, then the following holds:
	\begin{itemize}
		\item[(i)] If $b_n=0$, the $\mathcal{H}$-spectrum of $G$ coincides with the $\mathcal{H}$-spectrum of  $G(b_1, b_2, \ldots, b_{n-1})$;
		
		\item[(ii)] If $b_n=1$, the $\mathcal{H}$-spectrum of $G$ is $\{\lambda_1+1, \lambda_2+1, \ldots, \lambda_m+1, \mu_1+1, \mu_2+1, \ldots, \mu_{n-2}+1, n\}$
	\end{itemize}
\end{prop}

\begin{proof}
	Part (i) follows directly, as adding an isolated vertex does not affect the $\mathcal{H}$-spectrum.
	
	To consider (ii), we need to apply Proposition~\ref{lem-gp-1}. If $G_1$ of that proposition is an isolated vertex $K_1$ and $G_2$ is a threshold graph generated by  $(b_1, b_2, \ldots, b_{n-1})$, then the top-left block of $\mathcal{H}(G_1\vee G_2)$ reduces to an empty matrix, the middle block gives the $\mathcal{H}$-eigenvalues $\lambda_1+1, \lambda_2+1, \ldots, \lambda_m+1$, whereas the bottom-right block reduces to $A_{\overline{G}_2}+D_{G_2}+2I_{n-1}$, where $D_{G_2}$ is the diagonal matrix of vertex degrees of $G_2$. By computing $D_{G_2}=(n-2)I_{n-1}-D_{\overline{G}_2}$, and substituting $-L(\overline{G}_2)$ for $A_{\overline{G}_2}-D_{\overline{G}_2}$, we arrive at $nI_{n-1}-L(\overline{G}_2)$, which means that the remaining $\mathcal{H}$-eigenvalues are $n-\overline{\mu}_1, n-\overline{\mu}_2, \ldots, n-\overline{\mu}_{n-1}$, where $\overline{\mu}_i~(1\leq i\leq n-1)$ are the  non-increasingly indexed Laplacian eigenvalues of $\overline{G}_2$. However, since the Laplacian eigenvalues of any graph are related to the Laplacian eigenvalues of the complementary graph as $\overline{\mu}_i=n-1-\mu_{n-1-i}$, for $1\leq i\leq n-2$ \cite{mer-LAA}, we arrive the desired result, since $\overline{\mu}_{n-2}=0$.
\end{proof}

\begin{cor}
	Every threshold graph with at least one edge is $\mathcal{H}$-integral.
\end{cor}

\begin{proof}
	This result follows from the previous iterative computation, as  every threshold graph is Laplacian integral~\cite{Mer}.
\end{proof}

We provide an example.

\begin{examplex}
	Let $G$ be the threshold graph generated by $(0, 0, 1, 1, 0, 1)$, and denote by $G_i$ the graph generated by the first $i$ entries, for $3\leq i\leq 5$. Then, by Proposition~\ref{prop:thresh}, the $\mathcal{H}$-eigenvalues of $G_3$ are 3 and 1, as the unique Laplacian eigenvalue of $2K_1$ is 0 with multiplicity 2. Since the Laplacian eigenvalues of $G_3$ are 3, 1 and 0, we deduce that the $\mathcal{H}$-eigenvalues of $G_4$ are 4 with multiplicity 3 and 2  with multiplicity 2. Evidently, $G_5$ has the same $\mathcal{H}$-spectrum, whereas its Laplacian spectrum consists of 4 with multiplicity 2, along with 2 and 0 each with multiplicity~2. This gives
	\begin{equation*}
		{\Sp}(G) =
		\left\{\!\!\begin{array}{cccc}
			6 & 5 & 3 & 1 \\
			1 & 5 & 3 & 1
		\end{array}\right\}.
	\end{equation*}
\end{examplex}

Taking into account Propositions~\ref{H-gwg1}, \ref{H-Kn1nk} and~\ref{prop:thresh}, we obtain a contribution to Subsection~\ref{H47}.

\begin{cor}
The $\mathcal{H}$-nullity of a complete $k$-partite graph $K_{n_1, n_2, \ldots, n_k}$ is zero, unless $k=2$ and  $n_1=1$ or $n_2=1$. The $\mathcal{H}$-nullity of every generalized windmill graph and every threshold graph is zero.
\end{cor}

By a slight modification of the proof of Proposition~\ref{prop:thresh}, we arrive at the following iterative construction of $\mathcal{H}$-integral graphs.

\begin{prop}The following procedure generates $\mathcal{H}$-integral graphs:
	$$\left\{\begin{array}{ll}
		G_0, \text{a $\mathcal{H}$-integral graph}\\[2mm]
		G_{i+1}\cong K_s\vee (G_i\cup tK_1).
	\end{array}\right.$$
\end{prop}

\begin{proof} If  $G_i$ is simultaneously $\mathcal{H}$-integral and Laplacian integral, so is $G_{i+1}$ since adding of isolated vertices does not affect neither the former nor the latter integrality, and the same holds for adding dominating vertices (for the Laplacian spectrum see \cite{Mer}, and for the $\mathcal{H}$-spectrum see Proposition~\ref{prop:thresh}).
\end{proof}

We leave the following problem for further research.
\begin{problem}
Give some necessary conditions for a graph to be $\mathcal{H}$-integral, or even classify $\mathcal{H}$-integral graphs.
\end{problem}

\section{Further Motivation and Applications}\label{sec:rem}

For a wider context for our research,it is necessary to give a brief overlook on the classical Hodge theory and related disciplines. The classical Hodge theory \cite{Hodge} on  Riemannian manifolds is referred to a `differentiable Hodge theory' that confirms the existence and uniqueness of a harmonic form in every cohomology class. This theory  has been rapidly developed in the domain of topology, functional analysis and  geometry. As pointed out in \cite{Hodge-lim},  Hodge theories on metric spaces \cite{Hodge-metric} and simplicial complexes \cite{Hodge-sim} are recognized as the `continuous Hodge theory' and the `discrete Hodge theory', respectively. Remarkably, Huh, a Fields Medalist, has developed an idea of the Hodge theory in combinatorics \cite{huh-ICM}. Together with several collaborators, he has introduced the Hodge theory of matroids \cite{adi-huh-kat-0} and resolved the Heron-Rota-Welsh Conjecture for matroids \cite{adi-huh-kat}. In particular,  Lim \cite{Hodge-lim} has introduced the notion of the Hodge $k$-Laplacian $\Delta_k$ on graphs, where $0 \leq k \leq \omega$ and $\omega$ is the clique number of graphs. We omit the definition of $\Delta_k$ here as it requires further background (see \cite[Section 4]{Hodge-lim} for details); nevertheless, it serves as the starting point of our paper. In effect,  the graph Laplacian $\Delta_0$ and graph Helholtzian $\Delta_1$ defined in Section 1 are respectively known as the Hodge 0-Laplacian and Hodge 1-Laplacian, the first two lowest-order cases of Hodge Laplacians on graphs. Moreover, it seems that for $k\geq 2$ the graph matrix representation of $\Delta_k$ is not easy to obtain.

Regarding more specific motivations, we know that  the graphs appear as a natural mathematical model of simple networks whose edges  describe pairwise interactions between vertices. While, the (normalized) Laplacian matrix of a graph has been used as a usual pivotal tool in the study of structural and dynamical properties of such networks \cite{chung-bool1,chung-book}. However, as pointed out in \cite{sch-Siam},  ordinary graphs have certain obstacles in modelling many real-world networks such as social networks, brain neural networks, communication networks, or similar.  All these network types are generalized to the so-called {\it simplicial networks} with higher-order structures and interactions \cite{sch-Siam,shi-chen,tah-jad}. Simultaneously, the graph Laplacian is extended to the so-called {\it higher-order combinatorial Laplacian} by Eckmann \cite{eck}. Many details can be found in \cite{duv-kli-mar,hor-jos,kook-rei-sta,sch-Siam,ste-kli}.
Another higher-order generalization of the graph Laplacian is just the aforementioned Hodge Laplacian. It is worth mentioning that, in this model, a simplex is represented by a graph clique (i.e., a set of mutually adjacent vertices).  A Hodge $k$-Laplacian $\Delta_k$ on a graph generates graph matrices indexed by cliques of order $k+1$, where $0 \leq k \leq \omega(G)-1$ and $\omega(G)$ is the clique number of $G$. The Hodge 0-Laplacian gives the famous Laplacian matrix. The Hodge $1$-Laplacian $\Delta_1$ produces the Helmholtzian matrix, the main object we are dealing with in this paper, and both have been applied in networks \cite{chung-appl,muh-ege,sch-Siam} and statistical ranking \cite{jiang-lim}. However, it has attracted only very sporadic attention with respect to the Helmholtzian spectra of graphs.

Note that small-world networks contain abundant triangles \cite{shi-chen,wat-str}, and because the Helmholtzian matrix captures these triangles, its spectrum may prove especially useful in their analysis. Different graph spectra serve distinct roles for different problems. To that end, this paper, together with two earlier ones \cite{li-lu-wang,lu-shi-sta-ww}, provides a framework for developing a spectral theory based on the Helmholtzian matrix. These papers are foundational, and further work is anticipated. We also formulate several research problems that arose during our investigation.

\section*{Declaration of Interest Statement}

The authors declare that they have no known competing financial interests or personal relationships that could
have appeared to influence the work reported in this paper.

\section*{Data availability statement}

There is no associated data.

\section*{Acknowledgements}

Jianfeng Wang  would like to thank Professor  L.-H. Lim for his helpful suggestions, and Shu Li as well as Zhen Chen for their discussions.

Lu Lu is supported by the National Natural Science Foundation of China (No.~12371362). Yongtang Shi is supported by the National Natural Science Foundation of China (No.~11922112). Zoran Stani\'c is supported by the Ministry
of Science, Technological Development, and Innovation of the Republic of Serbia (No.~451-03-136/2025-
03/200104). Jianfeng Wang is supported by the National Natural Science Foundation of China (No.~12371353).  Yi Wang is supported by the National Natural Science Foundation of China (No.~12171002 and 12331012).

%The authors would like to thank the anonymous referees, who made many constructive suggestions that led to numerous improvements on the presentation of this paper.

{}
\end{document}